\documentclass[12pt]{amsart}
\usepackage{amsmath,amsthm}
\usepackage{amssymb}
\usepackage[all]{xypic}
\SelectTips{cm}{12}
\UseTips{}
\usepackage{euscript}
\begin{document}
%
%
\theoremstyle{plain}
\swapnumbers
    \newtheorem{thm}{Theorem}[section]
    \newtheorem{prop}[thm]{Proposition}
    \newtheorem{lemma}[thm]{Lemma}
    \newtheorem{cor}[thm]{Corollary}
    \newtheorem{fact}[thm]{Fact}
    \newtheorem{subsec}[thm]{}
    \newtheorem*{thma}{Theorem A}
\theoremstyle{definition}
    \newtheorem{assume}[thm]{Assumption}
    \newtheorem{defn}[thm]{Definition}
    \newtheorem{example}[thm]{Example}
    \newtheorem{examples}[thm]{Examples}
    \newtheorem{claim}[thm]{Claim}
    \newtheorem{notn}[thm]{Notation}
    \newtheorem{construct}[thm]{Construction}
\theoremstyle{remark}
        \newtheorem{remark}[thm]{Remark}
        \newtheorem{remarks}[thm]{Remarks}
    \newtheorem{ack}[thm]{Acknowledgements}
%
%
\newenvironment{myeq}[1][]
{\stepcounter{thm}\begin{equation}\tag{\thethm}{#1}}
{\end{equation}}
\newcommand{\mydiag}[2][]{\myeq[#1]\xymatrix{#2}}
\newcommand{\mydiagram}[2][]
{\stepcounter{thm}\begin{equation}
     \tag{\thethm}{#1}\vcenter{\xymatrix{#2}}\end{equation}}
%
\newenvironment{mysubsection}[2][]
{\begin{subsec}\begin{upshape}\begin{bfseries}{#2.}
\end{bfseries}{#1}}
{\end{upshape}\end{subsec}}
\newenvironment{mysubsect}[2][]
{\begin{subsec}\begin{upshape}\begin{bfseries}{#2\vsn.}
\end{bfseries}{#1}}
{\end{upshape}\end{subsec}}
\newcommand{\sect}{\setcounter{thm}{0}\section}
\newcommand{\wh}{\ -- \ }
\newcommand{\w}[2][ ]{\ \ensuremath{#2}{#1}\ }
\newcommand{\ww}[1]{\ \ensuremath{#1}}
\newcommand{\wb}[2][ ]{\ (\ensuremath{#2}){#1}\ }
\newcommand{\wref}[2][ ]{\ \eqref{#2}{#1}\ }
%
%
\newcommand{\xra}[1]{\xrightarrow{#1}}
\newcommand{\xla}[1]{\xleftarrow{#1}}
\newcommand{\xsim}{\xrightarrow{\sim}}
\newcommand{\hra}{\hookrightarrow}
\newcommand{\epic}{\to\hspace{-5 mm}\to}
\newcommand{\adj}[2]{\substack{{#1}\\ \rightleftharpoons \\ {#2}}}
\newcommand{\ccsub}[1]{\circ_{#1}}
\newcommand{\DEF}{:=}
\newcommand{\EQUIV}{\Leftrightarrow}
\newcommand{\hsp}{\hspace{10 mm}}
\newcommand{\hs}{\hspace{5 mm}}
\newcommand{\hsm}{\hspace{3 mm}}
\newcommand{\vsm}{\vspace{2 mm}}
\newcommand{\vsn}{\vspace{1 mm}}
\newcommand{\vs}{\vspace{4 mm}}
\newcommand{\rest}[1]{\lvert_{#1}}
\newcommand{\lra}[1]{\langle{#1}\rangle}
\newcommand{\llrr}[1]{\langle\!\langle{#1}\rangle\!\rangle}
\newcommand{\q}[1]{^{({#1})}}
\newcommand{\II}[2]{I^{#1}_{#2}}
\newcommand{\Ic}[1]{{\mathcal I}^{#1}}
\newcommand{\li}[1]{_{#1}}
\newcommand{\lb}[1]{_{({#1})}}
\newcommand{\uv}{\lb{u,v}}
%
%
\newcommand{\fd}{f_{\bullet}}
\newcommand{\phid}{\phi_{\bullet}}
\newcommand{\tfd}{\widetilde{\fd}}
\newcommand{\hfd}{\widehat{\fd}}
\newcommand{\gd}{g_{\bullet}}
\newcommand{\hgd}{\widehat{\gd}}
\newcommand{\tgd}{\widetilde{\gd}}
\newcommand{\hd}{h_{\bullet}}
\newcommand{\hhd}{\widehat{h}_{\bullet}}
\newcommand{\kd}{k_{\bullet}}
\newcommand{\hkd}{\widehat{\kd}}
\newcommand{\rd}{r_{\bullet}}
%
%
\newcommand{\A}{{\EuScript A}}
\newcommand{\bA}{\bar{\A}}
\newcommand{\hA}{h^{\A}}
\newcommand{\B}{{\EuScript B}}
\newcommand{\C}{{\mathcal C}}
\newcommand{\Ca}{\C_{\ast}}
\newcommand{\CA}{\C_{\A}}
\newcommand{\CK}{\C_{K}}
\newcommand{\CX}{\C_{X}}
\newcommand{\D}{{\mathcal D}}
\newcommand{\E}{{\mathcal E}}
\newcommand{\F}{{\EuScript F}}
\newcommand{\Fs}{F_{s}}
\newcommand{\K}{{\mathcal K}}
\newcommand{\M}{{\mathcal M}}
\newcommand{\N}{{\mathcal N}}
\newcommand{\OO}{{\EuScript O}}
\newcommand{\PP}{{\mathcal P}}
\newcommand{\PO}{{\mathcal PO}}
\newcommand{\RR}{{\mathcal R}}
\newcommand{\QQ}{{\mathcal Q}}
\newcommand{\Ss}{{\mathcal S}}
\newcommand{\Sa}{\Ss_{\ast}}
\newcommand{\TT}{{\mathcal T}}
\newcommand{\Ta}{\TT_{\ast}}
\newcommand{\U}{{\mathcal U}}
\newcommand{\V}{{\mathcal V}}
%
%
\newcommand{\hy}[2]{{#1}\text{-}{#2}}
\newcommand{\Ab}{{\EuScript Ab}}
\newcommand{\Abgp}{{\Ab\Grp}}
\newcommand{\Cat}{{\EuScript Cat}}
\newcommand{\CC}{\hy{\C}{\Cat}}
\newcommand{\DiG}{{\EuScript D}i{\EuScript G}}
\newcommand{\DG}{\D_{\Gamma}}
\newcommand{\Grp}{{\EuScript Gp}}
\newcommand{\Gpd}{{\EuScript Gpd}}
\newcommand{\JG}{J_{\Gamma}}
\newcommand{\RM}[1]{\hy{{#1}}{\EuScript Mod}}
\newcommand{\RG}{\RR_{\Gamma}}
\newcommand{\Set}{{\EuScript Set}}
\newcommand{\Sb}{S_{\cub}\,}
\newcommand{\UP}[1]{\U_p(#1)}
\newcommand{\VC}{\hy{\V}{\Cat}}
\newcommand{\WG}{W\Gamma}
\newcommand{\WGGp}{W(\Gamma,\Gp)}
\newcommand{\hWG}{\widehat{\WGGp}}
\newcommand{\FG}{\Fs\Gamma}
\newcommand{\WK}{W\K}
\newcommand{\X}{{\mathcal X}}
%
%
\newcommand{\CG}{(\C,\Gamma)}
\newcommand{\CGC}{\hy{\CG}{\Cat}}
\newcommand{\CaG}{(\Ca,\Gamma)}
\newcommand{\CaGC}{\hy{\CaG}{\Cat}}
\newcommand{\CoK}{(\C,\K)}
\newcommand{\CO}{(\C,\OO)}
\newcommand{\COC}{\hy{\CO}{\Cat}}
\newcommand{\OC}{\hy{\OO}{\Cat}}
\newcommand{\SG}{(\Ss,\Gamma)}
\newcommand{\SGC}{\hy{\SG}{\Cat}}
\newcommand{\SaG}{(\Sa,\Gamma)}
\newcommand{\SaGC}{\hy{\SaG}{\Cat}}
\newcommand{\SO}{(\Ss,\OO)}
\newcommand{\SOC}{\hy{\SO}{\Cat}}
\newcommand{\SeG}{(\Set,\Gamma)}
\newcommand{\SeGC}{\hy{\SeG}{\Cat}}
\newcommand{\VG}{(\V,\Gamma)}
\newcommand{\VGC}{\hy{\VG}{\Cat}}
\newcommand{\VO}{(\V,\OO)}
\newcommand{\VOC}{\hy{\VO}{\Cat}}
%
%
\newcommand{\Ed}{\E_{\bullet}}
\newcommand{\Md}{\M_{\bullet}}
\newcommand{\Nd}{\N_{\bullet}}
\newcommand{\Xd}{X_{\bullet}}
\newcommand{\Yd}{Y_{\bullet}}
%
%
\newcommand{\Q}{{\mathbb Q}}
\newcommand{\R}{{\mathbb R}}
\newcommand{\Z}{{\mathbb Z}}
%
%
\newcommand{\Gp}{\Gamma'}
\newcommand{\tG}{\tilde{\Gamma}}
\newcommand{\hX}{\widehat{X}}
\newcommand{\tX}{\tilde{X}}
\newcommand{\bz}{\mathbf{0}}
\newcommand{\bo}{\mathbf{1}}
\newcommand{\btw}{\mathbf{2}}
\newcommand{\bth}{\mathbf{3}}
\newcommand{\bfo}{\mathbf{4}}
\newcommand{\bk}{\mathbf{k}}
\newcommand{\bn}{\mathbf{n}}
\newcommand{\bnm}{\mathbf{n-1}}
\newcommand{\bnp}{\mathbf{n+1}}
%
%
\newcommand{\pis}{\pi_{\ast}}
\newcommand{\hpi}{\hat{\pi}_{1}}
%
%
\newcommand{\csk}[1]{\operatorname{csk}_{#1}}
\newcommand{\colim}{\operatorname{colim}}
\newcommand{\Colim}{\operatorname{Col}}
\newcommand{\Coli}[2]{\Colim\q{#1}({#2})}
\newcommand{\comp}[1]{\operatorname{comp}({#1})}
\newcommand{\cmp}[1]{\operatorname{cmp}_{({#1})}}
\newcommand{\cub}{\operatorname{cub}}
\newcommand{\seg}[1]{\operatorname{Seg}[{#1}]}
\newcommand{\sk}[1]{\operatorname{sk}_{#1}}
\newcommand{\skc}[1]{\operatorname{sk}^{c}_{#1}}
\newcommand{\cskc}[1]{\operatorname{cosk}^{c}_{#1}}
\newcommand{\skG}[1]{\skc{#1}(\Gamma,\Gp)}
\newcommand{\sktG}[1]{\skc{#1}(\tG,\Gp)}
\newcommand{\Dom}{\operatorname{Dom}}
\newcommand{\End}{\operatorname{End}}
\newcommand{\Ext}{\operatorname{Ext}}
\newcommand{\fib}{\operatorname{fib}}
\newcommand{\fin}{\operatorname{fin}}
\newcommand{\hc}[1]{\operatorname{hc}_{#1}}
\newcommand{\ho}{\operatorname{ho}}
\newcommand{\holim}{\operatorname{holim}}
\newcommand{\Hom}{\operatorname{Hom}}
\newcommand{\uHom}{\underline{\text{Hom}}}
\newcommand{\Id}{\operatorname{Id}}
\newcommand{\Image}{\operatorname{Im}}
\newcommand{\init}{\operatorname{init}}
\newcommand{\vf}{v_{\fin}}
\newcommand{\vfi}[1]{v{#1}_{\fin}}
\newcommand{\vi}{v_{\init}}
\newcommand{\vin}[1]{v{#1}_{\init}}
\newcommand{\Mor}{\operatorname{Mor}}
\newcommand{\Obj}{\operatorname{Obj}\,}
\newcommand{\op}{\sp{\operatorname{op}}}
\newcommand{\vdim}{\operatorname{vdim}}
%
%
%
\newcommand{\pgx}[1][n]{\mathcal P_{#1}(\Gamma,X')}
\newcommand{\qgx}[1][n]{\mathcal Q_{#1}(\Gamma,X')}
%
%
\newcommand{\map}{\operatorname{map}}
\newcommand{\mapp}{\map\,}
\newcommand{\mapa}{\map_{\ast}}
\newcommand{\mapc}{\C^{c}}
\newcommand{\mape}[2]{\map\sp{#1}\sb{#2}}
\newcommand{\mapcv}[1]{\mape{c}{#1}}
%
%
\newcommand{\bS}[1]{\mathbf{S}^{#1}}
\newcommand{\bSp}[2]{\bS{#1}_{({#2})}}
\newcommand{\bW}{\mathbf{W}}
%
%
\title{Higher homotopy operations and Cohomology}
%
%
\author[D.~Blanc]{David Blanc}
\address{Department of Mathematics\\ University of Haifa\\ 31905 Haifa\\ Israel}
\email{blanc@math.haifa.ac.il}
\author [M.W.~Johnson]{Mark W.~Johnson}
\address{Department of Mathematics\\  Penn State Altoona\\  Altoona,
PA 16601-3760\\ USA}
\email{mwj3@psu.edu}
\author[J.M.~Turner]{James M.~Turner}
\address{Department of Mathematics\\ Calvin College\\ Grand Rapids, MI\\ USA}
\email{jturner@calvin.edu}
\date{July 29, 2008; revised April 20, 2009}

\begin{abstract}
We explain how higher homotopy operations, defined topologically, may
be identified under mild assumptions with (the last of) the Dwyer-Kan-Smith
cohomological obstructions to rectifying homotopy-commutative diagrams.
\end{abstract}
\maketitle

\setcounter{section}{0}

%
%
\section*{Introduction}
\label{cint}

The first secondary homotopy operations to be defined were Toda
brackets, which appeared (in \cite{TodG}) in the early 1950's \wh at
about the same time as the secondary cohomology operations of Adem and
Massey (in \cite{AdemI} and \cite{MUehJ}). The definition was later extended to
higher order homotopy and cohomology operations (see
\cite{SpanH,MaunCC,KlausC}), which have been used extensively in
algebraic topology, starting with Toda's own calculations of the
homotopy groups of spheres in \cite{TodC}.

In \cite{BMarkH}, a ``topological'' definition of higher homotopy
operations based on the $W$-construction of Boardman and Vogt,
was given in the form of an obstruction theory for rectifying
diagrams. The same definition may be used also for higher cohomology
operations. This was recently modified in \cite{BChachP} to take
account of the fact that, in practice, higher order operations, both
in homotopy and in cohomology, occur in a \emph{pointed} context,
which somewhat simplifies their definition and treatment.

Earlier, in \cite{DKSmH}, Dwyer, Kan, and Smith gave an obstruction
theory for rectifying a diagram \w{\tX:\K\to\ho\TT} in the
homotopy category of topological spaces by making it
``infinitely-homotopy commutative'': the precise statement involves
the simplicial function complexes \w{\map(\tX u,\tX v)} for all
\w[,]{u,v\in\OO=\Obj(\K)} which constitute an \ww{\SO}-category
\w{\CX} (see \S \ref{dcx} and Section \ref{ccoh}). Their results are
thus stated in terms of \ww{\SO}-categories (simplicially enriched
categories with object set $\OO$). In particular, the obstructions
take values in the corresponding \ww{\SO}-cohomology groups (see
\cite[\S 2.1]{DKSmO}).

The purpose of the present note is to explain the relation between
these two approaches.  Because the $W$-construction, and thus higher
operations, are defined in terms of cubical sets, it is convenient to
work cubically throughout. In this language, \ww{\SO}-cohomology is
replaced by the (equivalent) \ww{\CG}-cohomology (see \S \ref{dcoh}),
and  our main result (Theorem \ref{tone} below) may be stated roughly
as follows:

Assume given a directed graph $\Gamma$ without loops (cf.\ \S \ref{dlatt})
of length \w[,]{n+2} having initial node \w{\vi} and terminal node
\w[,]{\vf}  and let $\M$ be a cubically enriched pointed model category.

\begin{thma}
For each pointed diagram \w[,]{\tX:\Gamma\to\ho\M} there is a natural
pointed correspondence $\Phi$ between the possible values of the
final Dwyer-Kan-Smith obstruction to rectifying $\tX$, in the
\ww{\CG}-cohomology group \w[,]{H^{n}(\Gamma,\pi_{n-1}\CX)} and the $n$-th
order homotopy operation \w[,]{\llrr{\tX}} a subset of
\w[.]{[\Sigma^{n-1}\tX(\vi),\,~\tX(\vf)]}
\end{thma}

\begin{remark}\label{rbb}
The fact that $\Phi$ is \emph{pointed} implies that, not surprisingly,
the two different obstructions to rectification vanish simultaneously.
Our objective here is to explicitly identify each value of a higher
homotopy operation (with its usual indeterminacy) with a \ww{\CG}-cohomology
class for $\Gamma$.

In \cite{BBlaC}, a relationship between \ww{\SO}-cohomology and the
cohomology of a $\Pi$-algebra is described. Since the latter is a
purely algebraic concept, we hope that together with the present result
this will provide a systematic way to apply homological-algebraic methods
to interpret and calculate higher homotopy and cohomology operations.
\end{remark}

\begin{mysubsection}{Notation}\label{snot}
The category of compactly generated topological spaces is denoted by $\TT$,
and that of pointed connected compactly generated spaces by \w[;]{\Ta} their
homotopy categories are denoted by \w{\ho\TT} and \w[,]{\ho\Ta}
respectively. The categories of (pointed) simplical sets will be denoted by
$\Ss$ (resp., \w[),]{\Sa} those of groups, abelian groups, and
groupoids by \w[,]{\Grp} \w[,]{\Abgp} and \w[,]{\Gpd} respectively.
\w{\Cat} denotes the category of small categories.

If \w{\lra{\V,\otimes}} is a monoidal category, we denote by \w{\VC}
the collection of all (not necessarily small) categories enriched over $\V$
(see \cite[\S 6.2]{BorcH2}). A category $\K$ is called \emph{pointed}
if it has a \emph{zero object} $0$ \wh that is, $0$ is both initial
and final. In such a $\K$, a map factoring through $0$  is
called a \emph{null} (or \emph{zero}) map, and since there is a unique
such map between any two objects, $\K$ is enriched over pointed sets.
\end{mysubsection}

\begin{remark}\label{rscat}
It will be convenient at times to work with \emph{non-unital}
categories \wh that is, categories which need not have identity maps.
These have been studied in the literature under various names,
beginning with the \emph{semi-categories} of V.V.~Vagner (see
\cite{VagnT}). The enriched version appears, e.g., in \cite{BBMoenR}.
\end{remark}

\begin{mysubsection}{Organization}\label{sorg}
Section \ref{ccs} provides a review of cubical sets and their homotopy
theory.  Section \ref{ccec} discusses cubically enriched categories,
as a replacement for the \ww{\SO}-categories of Dwyer and Kan, and
describes their model category structure (Theorem \ref{tcmodcat}).
In Section \ref{clhh} we give a ``topological'' definition of \emph{pointed}
higher homotopy operations in terms of diagrams indexed by certain
finite categories called lattices. Finally, in Section \ref{ccoh} the
Dwyer-Kan-Smith obstruction theory is described and the main result
(Theorem \ref{tone} and Corollary \ref{ccorresp}) is proved.
\end{mysubsection}

\begin{ack}
This research was supported by BSF grant 2006039; the third author was
also supported by NSF grant DMS-0206647 and a Calvin Research
Fellowship (SDG).
\end{ack}

%
%
\sect{Cubical sets}
\label{ccs}

Even though the obstruction theory of Dwyer, Kan, and Smith was
originally defined simplicially, for our purposes it appears more
economical to work cubically. This is because cubical sets are the
natural setting for the $W$-construction of Boardman and Vogt, which
was used for constructing higher homotopy operations in \cite{BMarkH}
and \cite{BChachP}. Since our goal is to identify these operations
with the cohomological obstructions of Dwyer-Kan-Smith, we simplify
the exposition by framing their theory in cubical terms as
well. Because cubical homotopy theory is less familiar than the
simplicial version, and the relevant information and definitions are
scattered throughout the literature, we summarize them here.

\begin{defn}\label{dcube}
Let $\Box$ denote the \emph{Box category},  whose objects are the
abstract cubes \w{\{\Ic{n}\}_{n=0}^{\infty}} (where \w{\Ic{}:=\{0,1\}}
and \w{\Ic{0}} is a single point). The morphisms of $\Box$ are
generated by the inclusions \w{d^{i}_{\varepsilon}:\Ic{n-1}\to\Ic{n}}
and projections \w{s^{i}:\Ic{n}\to\Ic{n-1}} for \w{1\leq i\leq n}
and \w[.]{\varepsilon\in\{0,1\}}

One can identify $\Box$ with a category of topological cubes, where
\w{\Ic{n}} corresponds to \w{[0,1]^{n}} (an $n$-fold product of unit
intervals), the linear map \w{d^{i}_{\varepsilon}:[0,1]^{n-1}\to[0,1]^{n}}
is defined
\w[,]{(t_{1},\dotsc,t_{n-1})\mapsto(t_{1},\dotsc,t_{i-1},\varepsilon,t_{i},
\dotsc,t_{n-1})} and \w{s^{i}:[0,1]^{n}\to [0,1]^{n-1}} is defined by
omitting the $i$-th coordinate.

A contravariant functor \w{K:\Box\op\to\Set} is called a \emph{cubical set}
(or cubical complex), and we write \w{K_{n}} for the set \w{K(\Ic{n})}
of \emph{$n$-cubes} (or $n$-cells) of $K$.
The \ww{(i,\varepsilon)}-\emph{face map}
\w{d_{i}^{\varepsilon}:K_{n}\to K_{n-1}} and the \emph{$i$-th degeneracy}
\w{s_{i}:K_{n-1}\to K_{n}} are induced by \w{d^{i}_{\varepsilon}} and
\w[,]{s^{i}} respectively. A cubical set $K$ is called \emph{finite} if
all but finitely many $n$-cubes of $K$ are degenerate (that is, in the
image of some \w[).]{s_{i}}
The category of cubical sets is denoted by $\C$. See \cite[I,\S 5]{KPortA},
\cite[\S 1]{BHigA}, or \cite{FRSandT}.

Several obvious constructions carry over from simplicial sets:
for example, the $n$-truncation functor \w{\tau_{n}} on cubical sets
has a left adjoint, and composing the two yields the cubical
$n$-\emph{skeleton} functor \w[.]{\skc{n}:\C\to\C} Thus \w{\skc{n}K}
is generated (under the degeneracies) by the $k$-cubes of $K$ for
\w[.]{k\leq n}
\end{defn}

\begin{notn}\label{ncube}
There is a standard embedding of $\Box$ in $\C$, in which
\w{\Ic{n}\in\Box} is taken to the standard $n$-cube \w{I^{n}\in\C}
(with one non-degenerate cell in dimension $n$, and all its faces).
Applying \w{\skc{n}} to the standard \ww{(n+1)}-cube \w[,]{I^{n+1}} we
obtain its \emph{boundary} \w[.]{\partial I^{n+1}:=\skc{n}I^{n+1}} By
omitting the \ww{d_{i}^{\varepsilon}}-face from \w[,]{\partial I^{n+1}}
we obtain the \ww{(i,\varepsilon)}-\emph{square horn}
\w[.]{\sqcap_{i}^{n,\varepsilon}}
\end{notn}

\begin{remark}\label{rcube}
There is also a version of cubical sets without degeneracies,
sometimes called \emph{semi-cubical sets}, but these are
not suitable for homotopy theoretic purposes (cf.\ \cite{AntoC}).
On the other hand, Brown and Higgins have proposed adding further
``adjacent degeneracies'', called \emph{connections} (see
\cite[\S 1]{BHigA} and \cite{GMaurC}).  These have proved useful in
various contexts (see, e.g., \cite{AntoG,BHigC}).
\end{remark}

\begin{mysubsection}{The cubical enrichment of $\C$}\label{smsc}
As a functor category, all limits and colimits in $\C$ are defined
levelwise. In particular, the $k$-cubes of a given cubical set \w{K\in\C}
\wb{k\geq 0} form a category \w{\CK} (under inclusions), and
\w[.]{K\cong\colim_{I^{k}\in\CK}\ I^{k}}

However, it turns out the products in $\C$ do not behave well with
respect to realization (see Remark \ref{rcsm} below), so another
monoidal operation is needed:
\end{mysubsection}

\begin{defn}\label{dctens}
If $K$ and $L$ are two cubical sets, their \emph{cubical tensor}
\w{K\otimes L\in\C} is defined
$$
K\otimes L~:=~\colim_{I^{j}\in\CK,~I^{k}\in\C_{L}}~I^{j+k}~.
$$
This defines a symmetric monoidal structure
\w{\otimes:\C\times\C\to\C} on cubical sets (see \cite[\S 3]{JardCH}).

More generally, let \w{\lra{\V,\otimes}} be a monoidal category with
(finite) colimits \wh for example, \w[,]{\lra{\TT,\times}}
\w[,]{\lra{\Ss,\times}} or \w{\lra{\C,\otimes}} \wh and assume we have
``standard cubes'' in $\V$, defined by a (faithful) monoidal functor
\w{T:\lra{\Box,\times}\to\lra{\V,\otimes}} \wh that is, a compatible choice
of ``standard cubes'' \w{T\Ic{n}} in $\V$. Given a (finite) cubical set
$K$, for any \w{X\in\V} define
$$
X\otimes K~:=~\colim_{\CK}~T_{X}~,
$$
where the diagram \w{T_{X}:\CK\to\V} is defined by
\w[.]{T_{X}I^{n}:=X\otimes T\Ic{n}}
\end{defn}

\begin{defn}\label{dcmap}
For \w{\lra{\V,\otimes}} as above, the \emph{cubical mapping complex}
\w{\mapcv{\V}(X,Y)\in\C} is defined for any \w{X,Y\in\V} by setting
$$
\mapcv{\V}(X,Y)_{n}~:=~\Hom_{\C}(X\otimes T\Ic{n},Y)~,
$$
with the cubical structure inherited from \w{\Ic{n}\in\Box}
(cf.\ \cite{KampsK}). We shall generally abbreviate
\w{\mapcv{\V}(X,Y)} to \w[.]{\V^{c}(X,Y)}
\end{defn}

In particular, when $\V$ is $\C$ itself, this makes
\w{\lra{\C,\otimes,I^{0},\mapc}} into a symmetric monoidal
closed category (see \cite[\S 6.1]{BorcH2}).

\begin{mysubsection}{Comparison to $\Ss$}\label{scss}
Cubical sets are related to simplicial sets by a pair of adjoint functors
%
\begin{myeq}\label{eqone}
\C\adj{T}{\Sb}\Ss~.
\end{myeq}
\noindent The \emph{triangulation} functor $T$ is defined
\w{TK:=\colim_{I^{n}\in\CK}~\Delta[1]^{n}} (compare Definition
\ref{dctens}), where \w{\Delta[1]^{n}=\Delta[1]\times\dotsc\times\Delta[1]}
is the standard simplicial $n$-cube.  The \emph{cubical singular} functor
\w{\Sb:\Ss\to\C=\Set^{\Box\op}} is defined adjointly by
\w[.]{(\Sb X)(I^{n}):=\Hom_{\Ss}(TI^{n},X)} This is a singular-realization pair
in the sense of \cite{DKanSR}; composing \wref{eqone} with the usual
adjoint pair:
%
\begin{myeq}\label{eqtwo}
\Ss\adj{|-|}{S}\TT
\end{myeq}
\noindent yields a similar adjunction to topological spaces.
\end{mysubsection}

\begin{remark}\label{rcsm}
Note that \w{T:\lra{\C,\otimes}\to\lra{\Ss,\times}} is strongly
monoidal (cf.\ \cite[\S 6.1]{BorcH2}), in that there is a natural isomorphism
%
\begin{myeq}\label{eqthree}
T(K\otimes L)~\cong~(TK)\times(TL)~.
\end{myeq}
\noindent On the other hand, \w{\Sb:\lra{\Ss,\times}\to\lra{\C,\otimes}}
is not strongly monoidal, as we now show:  as a right adjoint,
\w{\Sb} commutes with  (levelwise) products up to natural isomorphism, so
$$
\Sb(X\times Y)\cong\Sb(X)\times\Sb(Y)~.
$$
Thus, if \w{\Sb} were strongly monoidal, one would have
a levelwise isomorphism
$$
\Sb(X) \otimes \Sb(Y)\cong\Sb(X)\times\Sb(Y)~.
$$

Note this is unlikely, since an $n$-cube of \w{K\otimes L}
corresponds to a pair consisting of a $j$-cube of $K$ (for some
\w[)]{0 \leq j \leq n} and an \ww{(n-j)}-cube of $L$, while an
$n$-cube of \w{K\times L} corresponds to a pair consisting of an $n$-cube
of $K$ and an $n$-cube of $L$. In fact, \w{K\times L} is in general not even
homotopy equivalent to \w{K\otimes L} for \w[,]{K,L\in\C}
\wh for example, \w{T(I^{1}\times I^{1})\simeq S^{1}} in $\Ss$
while \w{T(I^{1}\otimes I^{1})\cong\Delta[1]\times\Delta[1]}
(see \cite[\S 1, Remark 8]{JardC}).

Nevertheless, since \w{I^{0}} is both terminal in $\C$ and the unit
for $\otimes$, the projections
\w{\pi_{K}:K\otimes L\to K\otimes I^{0}\cong K} and
\w{\pi_{L}:K\otimes L\to I^{0}\otimes L\cong L} induce a natural map
%
\begin{myeq}\label{eqfour}
\vartheta:K\otimes L~\to~K\times L~,
\end{myeq}
\noindent which is symmetric monoidal in the sense that it commutes
with the obvious associativity and switch-map isomorphisms.
\end{remark}
%
%
\begin{fact}[\protect{\cite[\S 3]{JardCH}}]\label{fmonic}
For any \w[,]{L\in\C} the functor \w{-\otimes L} preserves
monomorphisms in $\C$.
\end{fact}

\begin{remark}\label{rcact}
Note that \w{-\otimes I^{n}} preserves colimits, since it has a right
adjoint (defined by constructing the cubical set of maps between two
cubical sets as one does in $\Ss$ \wh see \cite[\S 4]{JardCH}).
Finally, observe that the cubical mapping complex for $\Ss$
(Definition \ref{dcmap}) is simply \w[.]{\mapcv{\Ss}(-,-)=\Sb\map_{\Ss}(-,-)}
\end{remark}
\begin{mysubsect}{The model category}\label{smcc}

Cubical sets were used quite early on as models for topological spaces \wh
see \cite{SerH}, \cite{EMacLAM}, \cite{MunkrS}, \cite{MacLHP},
\cite{PostCR,PostL}, and especially \cite{KanA1,KanA2}.  However, it
was Grothendieck,in \cite{GrotP}, who suggested that more generally
presheaf categories modeled on certain ``test categories'' $D$ can
serve as models for the homotopy category of topological spaces.
Cisinski, in his thesis \cite{CisiP}, carried out this program for
\w{D=\Box} (see also the exposition in \cite{JardCH}).
The model catgeory structure is very similar to the analogous one for
simplicial sets \wb[:]{D=\Delta}

\begin{defn}\label{dmcc}
A map \w{f:K\to L} in $\C$ is
\begin{enumerate}
\renewcommand{\labelenumi}{\alph{enumi})\ }
\item a \emph{weak equivalence} if \w{Tf:TK\to TL} is a weak
  equivalence in $\Ss$ (or equivalently, if \w{|Tf|} is a weak
  equivalence of topological spaces);
\item a \emph{cofibration} if it is a monomorphism.
\item a \emph{fibration} if it has the right lifting property (RLP) with
  respect to all acyclic cofibrations (i.e., those which are also weak
  equivalences) \wh that is, if in all commuting squares in $\C$:
\mydiagram[\label{rlp}]{
A \ar[r]^{g} \ar[d]_{i} & K \ar[d]^{f} \\
B \ar@{.>}[ru]^{\tilde{h}} \ar[r]_{h} & L
}
\noindent where $i$ is an acyclic cofibration, a map \w{\tilde{h}:B\to
  K} exists making the full diagram commute.
\end{enumerate}
\end{defn}

The model category defined here is proper, by \cite[Theorem 8.2]{JardCH}.
\end{mysubsect}

\begin{defn}\label{dscubes}
The \emph{cubical spheres} are \w{S^{n}:=\Sb(\Delta[n]/\partial\Delta[n])} for \w[,]{n\geq 1}
with the obvious basepoint. These corepresent the \emph{homotopy groups}
  \w[.]{\pi_{n}(-):=[S^{n},-]_{\ast}} Similarly,
  \w{S^{0}:=I^{0}\amalg\{\ast\}} corepresents \w[,]{\pi_{0}} and a map
  \w{f:K\to L} in \w{\Ca} is a weak equivalence if and only if it
  induces a \ww{\pi_{n}}-isomorphism for all \w[.]{n\geq 0}

  Note that we may define the fundamental groupoid \w{\hpi K} of an
  unpointed cubical set \w{K\in\C} as for simplicial sets or
  topological spaces (cf.\ \cite[Chapter 2]{HigC}).
\end{defn}

\begin{remark}\label{rmcc}
In analogy with the case of simplicial sets (see \cite[Ch. I]{GJarS})
one can show that cofibrations which are weak equivalences are the
same as the \emph{anodyne maps} \wh that is the closure of the set of
inclusions of the form
%
\begin{myeq}\label{eqfive}
i:\sqcap_{i}^{n,\varepsilon}\hra I^{n+1}
\end{myeq}
\noindent (see \S \ref{ncube}) under cobase change, retracts, coproducts, and
countable compositions (see \cite[\S 4]{JardCH}). Furthermore, the
fibrant objects and the fibrations in $\C$ can also be characterized
by Kan conditions \wh having the RLP with respect to maps of the form
\wref {eqfive} (see \cite{KanA1} and \cite[Theorem 8.6]{JardCH}).

As noted above (\S \ref{dcmap}), $\C$ is a symmetric monoidal closed
category (enriched over itself), with cubical mapping complexes
\w[.]{\mapc(-,-)} As shown in \cite[\S 3]{JardC}), it also satisfies the
cubical analogue of Quillen's Axiom SM7 (cf.\ \cite[II, \S 2]{QuiH}), so
$\C$ deserves to be called a \emph{cubical} model category. In
particular, if $L$ is a fibrant (Kan) cubical set, the function
complex \w{\mapc(K,L)} is fibrant, too, for any (necessarily
cofibrant) \w[.]{K\in\C}
\end{remark}

Finally, the following result shows that $\C$ indeed serves as a model
for the usual homotopy category of topological spaces:

%
%
\begin{prop}[Cf.\ \protect{\cite[Theorem 8.8]{JardCH}}]\label{pone}
The adjoint functors of \wref{eqone} induce equivalences of
homotopy categories \w{\ho\C\cong\ho\Ss} (so together with the pair
\wref[,]{eqtwo} we have \w[).]{\ho\C\cong\ho\TT}
\end{prop}

Note that since \w{I^{0}} is a final object in $\C$, the under
category \w{\Ca:=I^{0}/\C} of pointed cubical sets constitutes a pointed
version of $\C$, and we have:

%
%
\begin{fact}\label{fpointed}
There is a model category structure on \w[,]{\Ca} with the same weak
equivalences, fibrations, and cofibrations as $\C$.
\end{fact}

\begin{proof}
See \cite[Proposition 1.1.8]{HovM}.
\end{proof}

\begin{mysubsect}{Spherical model categories}\label{sfstr}

Like many other model categories, \w{\Ca} enjoys a collection of
additional useful properties that were axiomatized in \cite[\S 1]{BlaC}
under the name of a \emph{spherical} model category. This means that:
\begin{enumerate}
\renewcommand{\labelenumi}{(\alph{enumi})\ }
\item \w{\Ca} has a set $\A$ of \emph{spherical objects}: cofibrant
  homotopy cogroup objects (namely, the cubical spheres \w{\A=\{S^{n}\}_{n=1}^{\infty}}
  \wh Definition \ref{dscubes}). Furthermore, a map
  \w{f:K\to L} in \w{\Ca} is a weak equivalence if and only if
  \w{[A,f]} is an isomorphism for all \w[.]{A \in \A}
\item Each \w{K\in\Ca} has a functorial \emph{Postnikov tower}
of fibrations:
%
\begin{myeq}\label{eqsix}
\dotsc \to P_{n}K\xra{p\q{n}}P_{n-1}K\xra{p\q{n-1}}\dots\to P_{0}K~,
\end{myeq}
\noindent as well as a weak equivalence
\w{r:K\to P_{\infty}K:=\lim_{n}P_{n}K}
and fibrations \w{r\q{n}:P_{\infty}K\to P_{n}K} such that
\w{r\q{n-1}=p\q{n}\circ r\q{n}} for all $n$, and
\w{r\q{n}_{\#}:\pi_{k}P_{\infty}K\to\pi_{k}P_{n}K}
is an isomorphism for \w{k\leq n} and zero for \w[.]{k>n}
\item For every groupoid $\Lambda$, there is a functorial
  \emph{classifying object} \w{B\Lambda} with \w{B\Lambda\simeq
  P_{1}B\Lambda} and fundamental groupoid \w[,]{\hpi B\Lambda\cong \Lambda}
  unique up to homotopy.
\item Given a groupoid $\Lambda$ and a $\Lambda$-module $G$ (that is,
  an abelian group object over $\Lambda$), for each \w{n\geq 2} there
  is a functorial \emph{extended $G$-Eilenberg-Mac~Lane object}
  \w{E=E^{\Lambda}(G,n)} in \w[,]{\Ca/B\Lambda} unique up to homotopy,
  equipped with a section $s$ for
  \w[,]{(r\q{1}\circ r):E\to P_{1}E\simeq B\Lambda}
  such that \w{\pi_{n}E\cong G} as $\Lambda$-modules and \w{\pi_{k}E=0}
  for \w[.]{k\neq 0,1,n}
\item For every \w[,]{n\geq 1} there is a functor that
assigns to each \w{K\in\Ca} a homotopy pull-back square
%
\mydiagram[\label{eqseven}]{
\ar @{} [dr] |<<<{\framebox{\scriptsize{PB}}}
P_{n+1}K \ar[r]^{p\q{n+1}} \ar[d] &
P_{n}K \ar[d]^{k_{n}}\\ B\Lambda \ar[r] & E^{\Lambda}(M,n+2)
}
\noindent called an $n$-th \emph{$k$-invariant square} for $K$,
where \w[,]{\Lambda:=\hpi K} \w[,]{M:=\pi_{n+1}K} and
\w{p\q{n+1}:P_{n+1}K\to P_{n}K} is the given fibration of the
Postnikov tower.
\end{enumerate}

The map \w{k_{n}:P_{n}K\to E^{\Lambda}(M,n+2)} is called the $n$-th
(functorial) $k$-\emph{invariant} for $K$.
\end{mysubsect}

%
%
\begin{prop}\label{pcsphere}
The category \w{\Ca} is spherical.
\end{prop}

\begin{proof}
All the properties for \w{\Ca} follow from Fact \ref{fpointed}, and the
analogous results for \w{\Sa} or \w{\Ta} (see \cite[Theorem 3.15]{BJTurR}).
Note that homotopy groups for cubical sets appear in
\cite{KanA1,KanA2}, while (minimal, and thus non-functorial) Postnikov
towers for cubical sets were constructed by Postnikov in \cite{PostCR,PostL}.

For functorial cubical Postnikov towers, let the $n$-coskeleton
functor \w{\cskc{n}:\C\to\C} be the right adjoint to \w[,]{\skc{n}}
with \w{r\q{n}:\Id\to\cskc{n}} the obvious natural transformation, and
similarly for \w[.]{\Ca} By construction, \w{r\q{n}} is an isomorphism
in dimensions \w[.]{\leq n} If \w{K\in\Ca} is fibrant, so is
\w[,]{\cskc{n}K} and \w{\pi_{i}\cskc{n}K=0} for \w[,]{i>n} since
\w{\skc{n}S^{i}=\ast} for \w[.]{i>n} Thus if \w{K'\to K} is a
functorial fibrant replacement, and we change
\begin{myeq}\label{eqpostnikov}
K'\dotsc\to\cskc{n+1}K'\to\cskc{n}K'\to\cskc{n-1}K'\dotsc
\end{myeq}
\noindent functorially into a tower of fibrations, we obtain \wref[.]{eqsix}

For (strictly) functorial Eilenberg-Mac Lane objects, use
\cite[Prop.\ 2.2]{BDGoeR}, and apply \w[.]{\Sb}
For functorial $k$-invariants in \w[,]{\Ca} use the construction in
\cite[\S 5-6]{BDGoeR} (which works in \w[,]{\Ca} too).
\end{proof}

\begin{remark}\label{rcsk}
In general, the maps \w{\cskc{n}K\to\cskc{n-1}K} in
\wref{eqpostnikov} (adjoint to the inclusion of skeleta) are
\emph{not} fibrations (though the original construction of Kan,
when applied to a fibrant cubical set $K$, yields a tower of
fibrations with no further modification \wh see, e.g., \cite[VI, \S 2]{GJarS}).
However, if we are only interested in a specific Postnikov section
\w[,]{P_{n}} as long as $K$ is fibrant we can use \w{\cskc{n+1}K} as a
fibrant model for \w[,]{P_{n}K} and need only modify the next
section if we want \w{p\q{n+1}:P_{n+1}K\to P_{n}K} to be a fibration.
\end{remark}

%
%
\sect{Cubically enriched categories}
\label{ccec}

In \cite{DKanF}, Dwyer and Kan showed how any model category (more
generally, any small category $\M$ equipped with a class of weak
equivalences) can be enriched by simplicial function complexes, so
that the resulting simplicially enriched category encodes the homotopy
theory of $\M$ (see Remark \ref{rphhod} below).  Thus the category \w{s\Cat}
of simplicial small categories can be thought of as a ``universal
model category'', providing a setting for a ``homotopy theory of homotopy
theories''. Other such universal models were later provided in
\cite[\S 7]{DKSmH}, \cite{RezkM}, and \cite{BergT}.

An important subcategory of \w{s\Cat} consists of those simplicial
categories with a fixed set of objects. This is a special case of
the following:

\begin{defn}\label{dgvc}
For any set $\OO$, denote by \w{\OC} the category of all small
categories $\D$ with \w[.]{\OO:=\Obj\D}
More generally, assume \w{\Gamma\in\OC} is a small category, possibly
non-unital, and let \w{\lra{\V,\otimes}} be a monoidal category.
A \ww{\VG}-\emph{category} is a category \w{\D\in\OC} enriched over $\V$,
with mapping objects \w[,]{\mape{v}{\D}(-,-)\in\V} such that
%
\begin{myeq}\label{eqeight}
\Hom_{\Gamma}(u,v)=\emptyset~~~\Rightarrow~~~\mape{v}{\D}(u,v)~~~
\text{~is the initial object in~}~\V~.
\end{myeq}
Thus when $\V$ is pointed, we require \w{\mape{v}{\D}(u,v)=\ast}
whenever \w[.]{\Hom_{\Gamma}(u,v)=\emptyset}

The category of all \ww{\VG}-categories will be denoted
by \w[.]{\VGC} The morphisms in \w{\VGC} are enriched functors which
are the identity on $\OO$.

When \w{\Hom_{\Gamma}(u,v)} is never empty (so that we may disregard
condition \wref[)]{eqeight} we write \w{\VOC} instead of \w[.]{\VGC}
Dwyer and Kan call these  $\OO$-\emph{diagrams in} $\V$.
\end{defn}

\begin{remark}\label{rnonunital}
If $\Gamma$ is non-unital, \w{\Hom_{\Gamma}(u,u)} may be
empty, in which case \w{\mape{v}{\D}(u,u)} will be empty, if
\w{\V=\Set} or $\Ss$. This is allowed in the enriched version of
semi-categories (see Remark \ref{rscat}).  However, the discussion
below can be readily carried out in the context of ordinary (enriched)
categories, at the cost of paying attention to units. Thus if
$\V$ is pointed, \w{\Hom(u,u)} has (at least) two maps: the identity
and the zero map; these will coincide of $u$ is the zero object.
\end{remark}

We shall in fact concentrate on the case where $\Gamma$ has no self-maps
\w{u\to u} \wh e.g., a non-unital partially ordered set.
The main examples of \w{\lra{\V,\otimes}} to keep in mind are
\w[,]{\lra{\Set,\times}} \w[,]{\lra{\Grp,\times}} \w[,]{\lra{\Gpd,\times}}
\w[,]{\lra{\Ss,\times}} and \w[.]{\lra{\C,\otimes}}

\begin{mysubsect}{$\SO$-categories}\label{ssocats}

Although we shall be mainly concerned with \ww{\CG}-categories, we
first recall the more familiar simplicial version:

Note that when \w[,]{\V=\Ss} an \ww{\SG}-category can be thought of as
a simplicial object over \w{\OC} (or \w[).]{\SeGC}
Thus each \w{\Md\in\SOC} is a simplicial category with fixed object
set $\OO$ in each dimension, and all face and degeneracy functors are
the identity on objects (cf.\ \cite[\S 1.4]{DKanS}).
\end{mysubsect}

\begin{fact}\label{ffree}
The forgetful functor \w{U:\Cat\to\DiG} to the category of directed graphs
has a left adjoint \w[,]{F:\DiG\to\Cat} the \emph{free category} functor
(cf.\ \cite{HasseG}).
\end{fact}

\begin{defn}\label{dfree}
A simplicial category \w{\Ed\in\SOC}  is \emph{free} if each category
\w[,]{E_{n}} and each degeneracy functor \w[,]{s_{j}:E_{n}\to E_{n+1}}
is in the essential image of the functor $F$.

The pair of adjoint functors of Fact \ref{ffree} defines a comonad
\w[,]{FU:\Cat\to\Cat} and thus for each small category $\D$, an
augmented simplicial category \w{\Ed\to\D} with
\w[.]{\E_{n}:=(FU)^{n+1}\D} If \w[,]{\D\in\SeGC} then
\w[.]{\Ed\in\SGC} We denote this canonical \emph{free simplicial resolution}
of $\D$ by \w[.]{\Fs\D}
\end{defn}

\begin{remark}\label{rfree}
In \cite[\S 1]{DKanS}, Dwyer and Kan define a model category structure
on \w{\SOC} (also valid for \w[),]{\SGC} which turns out to be a
\emph{resolution model category} in the sense of \cite{BouC} (see also
\cite{Jar}, \cite[\S 5]{DKStE} and \cite[\S 2]{BJTurR}).
The spherical objects for \w{\SOC} (cf.\ \S \ref{sfstr}(a)) are
objects of the form \w{\Md:=\bSp{n}{u,v}} for \w{n\geq 1} and
\w[,]{\Hom_{\Gamma}(u,v)\neq\emptyset} defined by:
\begin{myeq}\label{eqsphere}
\M(u',v')=\begin{cases}\bS{n} & \text{for \ $u'=u$ \ and \ $v'=v$}\\
             \ast &  \text{otherwise,}\end{cases}
\end{myeq}
\noindent One can also show that \w{\SOC} and \w{\SGC} are spherical \wh
that is, endowed with the additional structure described in \S \ref{sfstr}
(of which only the existence of models is guaranteed in a resolution
model category).
\end{remark}

\begin{mysubsect}{The model category \ww{\CGC}}\label{smccgc}

In the case of \ww{\CG}-categories, the situation is somewhat
complicated by the fact that they cannot simply be viewed as cubical
objects in \w[,]{\Cat} because $\otimes$, and thus the composition
maps, are not defined dimensionwise (see Remark \ref{rcsm}).
Berger and Moerdijk have defined a model category structure for algebras
over coloured operads in a suitable symmetric monoidal model category,
which applies in particular to \w{\CGC} (see \cite{BMoerRC}, and
compare \cite{BMoerA}). However, in this paper we only need to
consider \ww{\CG}-categories for a special type of category $\Gamma$,
for which it is easy to describe an explicit model category structure
in which \w{\WG} is cofibrant:
\end{mysubsect}

\begin{defn}\label{dquasilat}
A small non-unital category $\Gamma$ will be called a
\emph{quasi-lattice} if it has no self-maps; in this case there is a
partial ordering on \w[,]{\OO=\Obj(\Gamma)} with \w{u\prec v} if and
only if \w[,]{\Hom_{\Gamma}(u,v)\neq\emptyset} and we require in
addition that $\Gamma$ be \emph{locally finite} in the sense that for
any \w{u\prec v} in $\OO$, the interval
\w{\seg{u,v}:=\{w\in\OO~|\ u\preceq w\preceq v\}} is finite.
\end{defn}

\begin{example}\label{eglatt}
The simplest example is a \emph{linear lattice} of length \w[,]{n+1} which we
denote by \w[:]{\Gamma_{n+1}} this consists of a single composable
\ww{(n+1)}-chain:
$$
\vi=(\bnp)~\xra{\phi_{n+1}}~\bn~\xra{\phi_{n}}~(\bnm)~\to~\dotsb~\to~\btw~
\xra{\phi_{2}}~\bo~\xra{\phi_{1}}~\bz=\vf~.
$$

Another example is a commuting square:
$$
\xymatrix@R=25pt{
\vi\ar[r]^{\phi'} \ar[d]_{\phi''} & v' \ar[d]^{\psi'} \\ v'' \ar[r]^{\psi''} & \vf
}
$$
\end{example}

Observe that for categories of diagrams indexed on a directed Reedy
category (i.e., one for which the ``inverse subcategory'' is trivial),
the Reedy model structure (cf.\ \cite[\S 15.2.2]{PHirM}) agrees with
the projective model structure. In this situation, cofibrations of
diagrams are those morphisms whose ``latching maps'' are all
cofibrations in the target category, while fibrations and weak
equivalences of diagrams are defined objectwise.

Our current context is sufficiently similar to allow an analogous
inductive  argument, depending on the following analog of Reedy's
latching objects and maps:

\begin{defn}\label{dcdiagram}
Given a quasi-lattice $\Gamma$, a map \w{F:\A\to\B} in \w[,]{\CGC}
and \w{u\prec v} in $\OO$, the \emph{composition category}
\w{(J^{\A,\B}\uv,<)} is a partially ordered set, whose objects are pairs
\w[,]{\lra{\omega,\X}} where $\omega$ is a chain
\w{\lra{u=w_{0}\prec w_{1}\prec \dotsc \prec w_{k-1}\prec w_{k}=v}} in
\w[,]{\lra{\OO,\prec}} and the index $\X$ is either $\A$ or $\B$.
We omit the copy of the trivial chain \w{\lra{u\prec v}} indexed by $\B$.

The partial order is defined by setting
\w{\lra{\omega,\X}\leq\lra{\omega',\X'}} whenever
\w{\omega'} is a (not necessarily proper) subchain of $\omega$, and
either \w{\X=\X'} or \w[,]{\X=\A}
\w[.]{\X'=\B}

The corresponding \emph{composition diagram}
\w{D=D^{\A,\B}\uv:J^{\A,\B}\uv\to\C} is defined by sending
\w{\lra{\omega,\X}} to \w[.]{\bigotimes_{j=1}^{k}~\X^{c}(w_{j-1},w_{j})}
The morphisms are generated by the following
two types of maps:
\begin{enumerate}
\renewcommand{\labelenumi}{(\roman{enumi})\ }
\item If \w{\omega'} is obtained from $\omega$ by omitting internal
  node \w{w_{j}} \wb[,]{1<j<k} the map
  \w{D\lra{\omega,\X}\to D\lra{\omega',\X}} is
\w[,]{\Id\otimes\dotsb\otimes
\cmp{w_{j-1},w_{j},w_{j+1}}^{\X}\dotsb\otimes\Id}
where
$$
\cmp{w_{j-1},w_{j},w_{j+1}}^{\X}:
\X^{c}(w_{j-1},w_{j})\otimes\X^{c}(w_{j},w_{j+1})~
\to~\X^{c}(w_{j-1},w_{j+1})
$$
is the cubical composition map in \w[;]{\X\in\{\A,\B\}}
\item The map \w{D\lra{\omega,\A}\to D\lra{\omega,\B}} is
  \w[.]{\bigotimes_{i=1}^{k}~F_{(w_{i-1},w_{i})}}
\end{enumerate}
\end{defn}

Note that \w[,]{F\uv:\A^{c}(u,v)\to\B^{c}(u,v)} together with the
composition maps of $\B$ ending in \w[,]{\B^{c}(u,v)} induce a map
\w[.]{\varphi\uv:\colim D^{\A,\B}\uv\to\B^{c}(u,v)}
In particular, when \w{\seg{u,v}=\{u,v\}} is minimal,
\w{\colim D^{\A,\B}\uv} is simply \w{\A^{c}(u,v)}
and \w{\varphi\uv} is \w[\vsm.]{F\uv:\A^{c}(u,v)\to\B^{c}(u,v)}

We now provide the details of the model category structure on
\w{\CGC} \wh inter alia, in order to allow the reader to verify that
the construction works in the non-unital setting:

%
%
\begin{lemma}\label{llimit}
If $\Gamma$ is a quasi-lattice, the category  \w{\CGC} has all limits
and colimits.
\end{lemma}

\begin{proof}
For any small category $\Gamma$, the limits in \w{\CGC} are constructed
by taking the limit at each
\w[,]{(u,v)\in\OO^{2}} with compositions defined for the product
\w{\prod_{i\in I}\,\A_{i}} by the obvious maps:
$$
(\prod_{i\in I}\,\A_{i}[u,w])\otimes(\prod_{i\in I}\,\A_{i}[w,v])~\to~
\prod_{i\in I}\,(\A_{i}[u,w]\otimes\A_{i}[w,v])~\xra{\cmp{u,w,v}}~
\prod_{i\in I}\,\A_{i}[u,v]~,
$$
and similarly for the other limits.

For the colimits, note that $\otimes$ is defined as a colimit (cf.\ Definition
\ref{dctens}), so it commutes with colimits in $\C$.
For \ww{\CG}-categories \w[,]{\{\A_{i}\}_{i\in I}} the coproduct
\w{\D:=\coprod_{i\in I}\,\A_{i}} is defined by induction on
the cardinality of \w{\seg{u,v}} in \w[.]{\lra{\OO,\prec}}
When \w{\seg{u,v}=\{u,v\}} is minimal, we let
\w[.]{\D(u,v):=\coprod_{i\in I}\,\A_{i}(u,v)} In general, set
$$
\D(u,v)~:=~\coprod_{i\in I}\,\A_{i}(u,v)~\amalg~
\coprod_{u\prec w\prec v}\D(u,w)\otimes\D(w,v)~,
$$
\noindent with the obvious (tautological) composition on the
right-hand summands\vsm.

Now given maps \w{F:\A\to\B} and \w{G:\A\to\E} in \w[,]{\CGC} the
pushout \w{\PO} is once more defined by induction on the cardinality of
\w[,]{\seg{u,v}} as follows:

In the initial case, when \w{\seg{u,v}} is minimal, \w{\PO(u,v)} is simply the pushout of
\w{\E(u,v)\leftarrow\A(u,v)\to\B(u,v)} in $\C$.

In the induction step, we let \w{J=J^{\PO}\uv} denote the union of the composition categories
\w[,]{J^{\A,\B}\uv} \w[,]{J^{\A,\E}\uv} and \w{J^{\B,\PO}\uv}
(see Definition \ref{dcdiagram}). Thus the objects of $J$ are pairs
\w[,]{\lra{\omega,\X}} where $\omega$ is a chain
\w{\lra{u=w_{0}\prec w_{1}\prec\dotsc w_{k}=v}} and
\w[,]{\X\in\{\A,\B,\E,\PO\}} again omitting \w[.]{\lra{u\prec v,\PO}}
Again $J$ is a partially ordered set, with the order relation defined
to be the union of those for \w[,]{(\A,\B)} \w[,]{(\A,\E)} and \w[.]{(\B,\PO)}

The composition diagrams \w[,]{D^{\A,\B}\uv} \w[,]{D^{\A,\E}\uv}
\w[,]{D^{\B,\PO}\uv} and \w{D^{\E,\PO}\uv} fit together to form a
composition diagram \w[.]{D^{\PO}\uv:J^{\PO}\uv\to\C}
The last two diagrams are well-defined, because we omit the
trivial chain \w[,]{\lra{u\prec v,\PO}} and all other values of
\w{D^{\B,\PO}\uv} and \w{D^{\E,\PO}\uv} have already been defined by
our induction assumption. We now let \w{\PO(u,v)} be the colimit in
$\C$ of the diagram \w[.]{D^{\PO}\uv:J^{\PO}\uv\to\C}

The constructions of the coproducts and pushouts implies that all
colimits exist in \w[,]{\CGC} by the dual of
\cite[Thm.\ 2.8.1 \& Prop.\ 2.8.2]{BorcH1}.
\end{proof}

\begin{defn}\label{dcmodcat}
Let $\Gamma$ be a quasi-lattice, and let $\A$ and $\B$ be
\ww{\CG}-categories. A map \w{F:\A\to\B} in \w{\CGC} is
\begin{enumerate}
\renewcommand{\labelenumi}{(\alph{enumi})\ }
\item a \emph{weak equivalence} if \w{F\uv:\A^{c}(u,v)\to\B^{c}(u,v)} is a
  weak equivalence in $\C$ (see \S \ref{smcc}) for any \w{u\prec v} in $\OO$.
\item a \emph{fibration} if \w{F\uv:\A^{c}(u,v)\to\B^{c}(u,v)} is a (Kan)
   fibration in $\C$ for all \w{u\prec v} in $\OO$.
\item a \emph{(acyclic) cofibration} if for all \w{u\prec v} in
  $\OO$ the maps \w{F\uv:\A^{c}(u,v)\to\B^{c}(u,v)} and
  \w{\varphi\uv:\colim D^{\A,\B}\uv\to\B^{c}(u,v)} are (acyclic)
  cofibrations in $\C$.
\end{enumerate}
\end{defn}

\begin{remark}\label{racycof}
A straightforward induction shows that the acyclic cofibrations so
defined are precisely those cofibrations which are weak equivalences.
\end{remark}

The following lemmas show that these choices yield a model category
structure on \w[:]{\CGC}

%
%
\begin{lemma}\label{lllp}
If $\Gamma$ is a quasi-lattice, \w{F:\A\to\B} is a cofibration and
\w{P:\D\to\E} is an fibration in \w[,]{\CGC} and either $F$ or $P$ is
a weak equivalence, then there is a lifting
$\hat{H}$ in any commutative square
%
\mydiagram[\label{eqnine}]{
\A \ar[r]^{G} \ar[d]_{F} & \D \ar[d]^{P} \\
\B \ar[r]_{H}  \ar@{.>}[ru]^{\tilde{H}} & \E~.
}
\end{lemma}

\begin{proof}
We choose \w{\tilde{H}\uv:\B^{c}(u,v)\to\D^{c}(u,v)} by induction on
the cardinality of the interval \w{\seg{u,v}} in \w[:]{\lra{\OO,\prec}}

When \w{\seg{u,v}=\{u,v\}} is minimal, we simply choose a lift
  \w{\tilde{H}\uv} in:
$$
\xymatrix@R=25pt{
\A^{c}\uv \ar[rr]^{G\uv} \ar[d]_{F\uv} & & \D^{c}\uv \ar[d]^{P\uv} \\
\B^{c}\uv \ar[rr]_{H\uv}  \ar@{.>}[rru]^{\tilde{H}\uv} & & \E^{c}\uv
}
$$
using the fact that \w{F\uv} is a cofibration and \w{P\uv} an acyclic
fibration in $\C$ (see \wref{rlp} above).

In the induction step, assume we have chosen compatible lifts \w{\tilde{H}\lb{u',v'}} for all
proper subintervals \w[.]{\seg{u',v'}\subset\seg{u,v}} These yield a
map $\hat{G}$ making the following solid square commute in $\C$:
$$
\xymatrix@R=25pt{
\colim D^{\A,\B}\uv \ar[d]_{\varphi\uv} \ar[rr]^{\hat{G}} & & \D^{c}\uv \ar[d]^{P\uv} \\
\B^{c}(u,v) \ar[rr]_{H\uv}  \ar@{.>}[rru]^{\tilde{H}\uv} & & \E^{c}\uv
}
$$
\noindent and since \w{\varphi\uv} is a cofibration by
Definition \ref{dcmodcat}, and \w{P\uv} is an acyclic fibration by
assumption, the lifting \w{\tilde{H}\uv} exists.

The same argument shows that there exists a lifting in \wref{eqnine}
when \w{F:\A\to\B} is an acyclic cofibration and \w{P:\D\to\E} is a
fibration.
\end{proof}

%
%
\begin{lemma}\label{lfact}
If $\Gamma$ is a quasi-lattice, any map \w{F:\A\to\B} in \w{\CGC} factors as:
%
\mydiagram[\label{eqten}]{
\A \ar[rd]^{I} \ar[rr]^{F} & & \B \\
& \D \ar[ru]_{P} &
}
\noindent where $I$ is a cofibration and $P$ is a fibration; and we can require
either $I$ or $P$ to be a weak equivalence.

\end{lemma}

\begin{proof}
Again construct $\D$, $I$, and $P$ in \wref{eqten} by induction on the
cardinality of \w[.]{\seg{u,v}} When \w{\seg{u,v}=\{u,v\}} is minimal,
choose any factorization:
$$
\xymatrix@R=25pt{
\A^{c}\uv \ar[rd]^{I\uv} \ar[rr]^{F\uv} & & \B^{c}\uv \\
& \D^{c}\uv \ar[ru]_{P\uv} &
}
$$
where \w{I\uv} is an acyclic cofibration and \w{P\uv} is a fibration
in $\C$\vsm.

Now assume by induction that we have chosen compatible factorizations
%
\mydiagram[\label{eqtwelve}]{
\A^{c}\lb{u,w}\otimes\A^{c}\lb{w,v} \ar[rr]^{\zeta^{\A}\uv}
\ar[d]_{I\lb{u,w}\otimes}\ar[d]^{I\lb{w,v}} &&
\Coli{\A}{u,v} \ar[rr]^{\omega^{\A}} \ar[d]^{\phi\uv} &&
\A^{c}\uv \ar[d]^{\eta\uv} \\
\D^{c}\lb{u,w}\otimes\D^{c}\lb{w,v} \ar[rr]^{\zeta^{\D}\uv}
\ar[d]_{P\lb{u,w}\otimes}\ar[d]^{P\lb{w,v}} &&
\Coli{\D}{u,v} \ar[rr]^{\theta\uv} \ar[d]^{\psi\uv} &&
\PO\uv \ar[d]^{\xi\uv} \\
\B^{c}\lb{u,w}\otimes\D^{c}\lb{w,v} \ar[rr]^{\zeta^{\B}\uv} &&
\Coli{\B}{u,v} \ar[rr]^{\omega^{\B}} && \B^{c}\uv
}
\noindent where each cubical set \w{\Coli{\E}{u,v}} (for
\w[)]{\E=\A,\B,\D} is the colimit over all proper subintervals
\w{\seg{u',v'}\subset\seg{u,v}}  and \ \w{u'\prec w\prec v'} of the diagram of
composition maps
$$
\w{\cmp{u',w,v'}^{\E}:\E^{c}(u',w)\otimes\E^{c}(w,v')~\to~\E^{c}(u',v')}.
$$
\noindent in each row,
\w{\zeta^{\E}:\E^{c}\lb{u,w}\otimes\E^{c}\lb{w,v}\to\Coli{\E}{u,v}}
is the structure map for the colimit, while
\w{\omega^{\E}:\Coli{\E}{u,v}\to\E^{c}\lb{u,v}} is induced by the
compositions. The cubical set \w{\PO\uv} is the pushout of the upper
right-hand square, with structure maps \w{\eta\uv} and \w[,]{\theta\uv}
and \w{\xi\uv} is induced on the pushout by \w{F\uv} and the maps
\w{\psi\uv} (from the naturality of the colimit) and
\w[.]{\omega^{\B}}

Note that the map \w{I\lb{u,w}\otimes I\lb{w,v}} is an acyclic
cofibration in $\C$ (see Fact \ref{fmonic}), so the induced map
\w{\phi\uv} is, too, as is \w[,]{\eta\uv} by cobase change. The map
\w[,]{P\lb{u,w}\otimes P\lb{w,v}} as well as the induced map
\w[,]{\psi\uv} comes from the compatible factorizations \wref[.]{eqtwelve}

Finally, choose a factorization
$$
\xymatrix@R=25pt{
\PO(u,v) \ar[rd]^{\zeta\uv} \ar[rr]^{\xi\uv} & & \B^{c}(u,v) \\
& \D^{c}(u,v) \ar[ru]_{P\uv} &
}
$$
\noindent where \w{\zeta\uv} is an acyclic cofibration and \w{P\uv} is a
fibration in $\C$. This defines the cubical set \w[,]{\D^{c}(u,v)}
which is equipped with composition maps
$$
\cmp{u,w,v}:=\zeta\uv\circ\theta\uv\circ\zeta^{\D}\uv~.
$$
Setting \w{I\uv:=\zeta\uv\circ\eta\uv} yields the required acyclic
cofibration, and since \w{\xi\uv} is induced by $F$, we have
\w[,]{P\uv\circ I\uv=F\uv} as required\vsm.

The same construction, \emph{mutatis mutandis}, yields a factorization
\wref{eqten} where $I$ is a cofibration and $P$ an acyclic fibration.
\end{proof}

%
%
\begin{thm}\label{tcmodcat}
If $\Gamma$ is a quasi-lattice, Definition \ref{dcmodcat} provides a
model category structure on \w[.]{\CGC}
\end{thm}

\begin{proof}
The category \w{\CGC} is complete and cocomplete by Lemma \ref{llimit}.
The classes of weak equivalences and fibrations are clearly closed under
compositions, and include all isomorphisms.  The same holds for cofibrations by
an induction argument.  Also, if two out of the three maps
$F$, $G$, and \w{G\circ F} are weak equivalences, so is the third.
The lifting properties for (co)fibrations are in Lemma \ref{lllp}, and the
factorizations are given by Lemma \ref{lfact}.
\end{proof}

As expected, the two key types of \ww{\VG}-categories are related by
suitable functors (compare \cite[Prop.~6.4.3]{BorcH2}):

%
%
\begin{prop}\label{padjenr}
For any quasi-lattice $\Gamma$, the functors \w{T:\C\to\Ss} and
\w{\Sb:\Ss\to\C} of \wref{eqone} extend to functors \w[.]{\CGC\adj{}{}\SGC}
Futhermore, this is a strong Quillen pair (cf.\ \cite[\S 8.5.1]{PHirM}),
and descends to an adjunction at the level of homotopy categories.
\end{prop}

\begin{proof}
The functor $T$ extends to \w{\CGC} by \wref[.]{eqthree} For
\w[,]{\Sb} given \w[,]{\A^{s}\in\SGC} with composition
\w{\xi:\A^{s}(u,w))\times \A^{s}(w,v)\to\A^{s}(u,v)}
we define the composition map
\w{\cmp{u,w,v}:\Sb(\A^{s}(u,w))\otimes\Sb(\A^{s}(w,v))\to\Sb(\A^{s}(u,v))}
for the \ww{\CG}-category \w{\Sb\A^{s}} to be the composite
\w{\Sb\xi\circ\vartheta} (see \wref[).]{eqfour}
As $\Sb$ is a strong right Quillen functor, it follows from the
definitions that the extension is also strong right Quillen.
\end{proof}

\begin{mysubsect}{Semi-spherical structure on \ww{\CGC}}\label{ssscgc}

The discussion above, including the model category structures, is
valid when we replace $\C$ or $\Ss$ by their pointed versions
(see \cite[Proposition 1.1.8]{HovM}). Moreover, even though we cannot
construct entry-wise spheres for \ww{\CG}-categories as in
\wref[,]{eqsphere} the category \w{\CGC} may be called
\emph{semi-spherical}, in the sense of having the rest of the spherical
structure described in \S \ref{sfstr}, as follows:
\end{mysubsect}

\begin{defn}\label{dpi}
Given a quasi-lattice $\Gamma$ and a \ww{\CaG}-category $\A$, its
\emph{fundamental groupoid} is the \ww{(\Gpd,\Gamma)}-category
obtained by applying the fundamental groupoid functor \w{\hpi}
to $\A$. Note that because \w{\hpi:\C\to\Gpd} factors through
\w[,]{T:\C\to\Ss} using \wref{eqthree} we see that \w{\hpi\A} is
indeed a \ww{(\Gpd,\Gamma)}-category (cf.\ \cite[Prop.~6.4.3]{BorcH2}).

Similarly, for each \w{n\geq 2} the functor \w[,]{\pi_{n}} applied
entrywise to $\A$, yields a \ww{(\Grp,\Gamma)}-category, which is actually
a \ww{(\RM{\hpi\A},\Gamma)}-category (see Definition \ref{dgvc}).
Note that, as for topological spaces, \w{\pi_{n}\A} is a module over
\w[.]{\hpi\A}
\end{defn}

\begin{enumerate}
\renewcommand{\labelenumi}{\alph{enumi})\ }
\item Each \ww{\CaG}-category $\A$ has a functorial Postnikov tower,
  obtained by applying the functors \w{P_{n}} of \S \ref{eqsix}
  to each \w[,]{\A^{c}(u,v)} and using
  \begin{equation*}
  \begin{split}
  P_{n}(\A(u,v)) \otimes P_{n}(\A(v,w)) &\to P_{n}(P_{n}(\A(u,v)) \otimes P_{n}(\A(v,w))) \\
    &\cong P_{n}(\A(u,v) \otimes \A(v,w)) \to P_{n}(\A(u,w)) \ .
   \end{split}
  \end{equation*}
\item For every \ww{(\Gpd,\Gamma)}-category $\Lambda$, there is a
  functorial \emph{classifying object} \w[.]{B\Lambda\in\CaGC}
\item Given a \ww{(\Gpd,\Gamma)}-category $\Lambda$, and
  a $\Lambda$-module $G$ (i.e., an abelian group object in
  \w[),]{\SeGC/\Lambda} for each \w{n\geq 2} there
  is a functorial \emph{extended $G$-Eilenberg-Mac~Lane object}
  \w{E^{\Lambda}(G,n)} in \w[.]{\CaGC/B\Lambda}
\item For  \w[,]{n\geq 1} there is is a functorial $k$-invariant
  square for $\A$ as in \wref[.]{eqseven}
\end{enumerate}

All these properties are straightforward for \w{\SaGC} (by applying
the analogous functors for \w{\Sa} componentwise), and they may
be transfered to \w{\CaGC} using Proposition \ref{padjenr}.

\begin{defn}\label{dcoh}
Given a \ww{(\Gpd,\Gamma)}-category $\Lambda$, a
$\Lambda$-module $G$, a \ww{\CaG}-category $\A$, and a \emph{twisting map}
\w[,]{p:\A\to B\Lambda} we define the $n$-th \ww{\CG}-\emph{cohomology
group} of $\A$ with coefficients in $G$ to be
$$
H^{n}_{\Lambda}(\A,G)~:=~[\A,E^{\Lambda}(G,n)]_{\CGC/B\Lambda}.
$$
\end{defn}

\begin{remark}\label{rcoh}
Typically, we have \w[,]{\Lambda=\hpi\A} with the obvious map $p$.

More generally, in \cite{DKSmO} Dwyer, Kan, and Smith give a definition
of the \ww{\SO}-cohomology of any \ww{\SO}-category with coefficients in a
$\Lambda$-module $G$; and there is also a relative version, for a pair
\w{(\A,\B)} (cf.\ \cite[\S 2.1]{DKSmO}). It is straightforward to verify
that the two definitions of cohomology coincide (when they are both
defined) under the correspondence of Proposition \ref{padjenr}.
\end{remark}

%
%
\sect{Lattices and higher homotopy operations}
\label{clhh}

We can now define higher homotopy operations as obstructions
to rectifying a homotopy commutative diagram \w[,]{X:\K\to\ho\TT}
using the approach of \cite{BMarkH}, with the modification in the
pointed case given in \cite{BChachP}. For this purpose, it is
convenient to work with a specific cofibrant cubical resolution of the
indexing category $\K$.  We need make no special assumptions about
$\K$ at this stage.

Boardman and Vogt originally defined their ``bar construction'' \w{\WK}
topologically (see \cite[III, \S 1]{BVogHI}). The \ww{\CO}-version may be
described as follows:

\begin{defn}\label{dwconst}
The \emph{W-construction} on a small category $\K$ with \w{\OO=\Obj\K} is the
\ww{\CO}-category \w[,]{\WK} with the cubical mapping complex
\w{\WK(a,b)} for every \w[,]{a,b\in\Obj(\K)} constructed as follows:

For every composable sequence
\begin{myeq}\label{eqcompseq}
\w{\fd=(a=a_{n+1}\xra{f_{n+1}}a_{n}\xra{f_{n}}a_{n-1}\dotsc a_{1}
\xra{f_{1}}a_{0}=b)}
\end{myeq}
\noindent of length \w{n+1} in $\K$, there is an $n$-cube \w{\II{n}{\fd}}
in \w[,]{\WK(a,b)} subject to two conditions:

\begin{enumerate}
\renewcommand{\labelenumi}{(\alph{enumi})~}
\item The $i$-th $0$-face of \w{\II{n}{\fd}} is identified with
  \w[,]{\II{n-1}{f_{1}\circ\dotsc\circ (f_{i}\cdot f_{i+1})\circ\dotsc f_{n+1}}}
  that is, we carry out the $i$-th composition in the sequence \w{\fd} (in the
  category $\K$).
\item The cubical composition
$$
\WK(a_{0},a_{i})\otimes\WK(a_{i},a_{n+1})~\to~\WK(a_{0},a_{n+1})=\WK(a,b)
$$
identifies the ``product'' \ww{(n-1)}-cube
\w{\II{i}{f_{0}\circ\dotsc\circ f_{i}}\otimes
       \II{n-i-1}{f_{i+1}\circ\dotsc\circ f_{n+1}}}
with the $i$-th $1$-face of \w[.]{\II{n}{\fd}}
\end{enumerate}
\end{defn}

\begin{notn}\label{ncompose}
Note the three different kinds of composition that occur in \w[:]{\WK}

\begin{enumerate}
\renewcommand{\labelenumi}{(\alph{enumi})~}
\item The \emph{internal} composition of $\K$ is denoted by
  \w[,]{f\cdot g} or simply \w[.]{fg}
\item The \emph{cubical} composition of \w[,]{\WK} denoted by
  \w[,]{f\otimes g} which corresponds to the $\otimes$-product of the
  associated cubes.
\item The \emph{potential} composition of \w[,]{\WK} denoted by
   \w[,]{f\circ g} is the heart of the $W$-construction: it provides
   another dimension in the cube for the homotopies between \w{f\otimes g}
   and \w[.]{f\cdot g}
\end{enumerate}
Thus a composable sequence \w{\fd} as in \wref{eqcompseq} (indexing a cube
in \w[)]{\WK} will be denoted in full by \w[;]{f_{1}\circ\dotsc\circ f_{n+1}}
the composed map \w{f_{1}f_{2}\dotsb f_{n+1}:a\to b} in $\K$ is denoted by
\w[;]{\comp{\fd}} and the cubical composite
\w{f_{1}\otimes f_{2}\otimes\dotsb\otimes f_{n+1}} will be denoted by
\w{\otimes\fd} (as an index for a suitable cube in \w[).]{\WK}
\end{notn}

\begin{defn}\label{dvertices}
The \emph{minimal} vertex of \w{\II{n}{\fd}} is
\w[,]{\II{0}{\comp{\fd}}} which is in the image of all $0$-face maps.
The opposite \emph{maximal} vertex, in the image of all $1$-face maps, is indexed
by \w{\otimes{\fd}} according to the convention above, with
\w{\II{0}{f_{1}}\otimes\II{0}{f_{2}}\otimes\dotsb\otimes\II{0}{f_{n+1}}}
identified with \w{\II{0}{\otimes\fd}} under the iterated cubical compositions.
\end{defn}

If we think of a small category $\K$ as a constant cubical category in
\w{\COC} for \w[,]{\OO=\Obj\K} there is an obvious map of \ww{\CO}-categories
\w[,]{\gamma^{c}:\WK\to\K} and following work of \cite{Lei} and \cite{Cord} we show:

%
%
\begin{lemma}\label{lwkresol}
The map \w{T\gamma^{c}:T\WK\to T\K=\K} may be identified with
\w{\gamma^{s}:\Fs\K\to\K} (see \S \ref{scss} \textit{ff.}).
\end{lemma}

\begin{proof}
Consider an individual cube  \w{\II{n}{\phid}} of \w[:]{\WK}
this is isomorphic to \w[,]{\WG_{n+1}}  where \w{\Gamma_{n+1}}
(Example \ref{eglatt}) consists of a composable sequence of \w{n+1} maps:
$$
(\bnp)~\xra{\phi_{n+1}}~\bn~\xra{\phi_{n}}~(\bnm)~\to~\dotsb~\to~\btw~
\xra{\phi_{2}}~\bo~\xra{\phi_{1}}~\bz~.
$$
\noindent The free simplicial resolution of \w{\FG_{n+1}} is
the triangulation of the $n$-cube \w{\II{n}{\phid}} by \w{n!}
$n$-simplices, corresponding to the possible full parenthesizations
of \w{\phid} (see Figure \ref{fig1}).

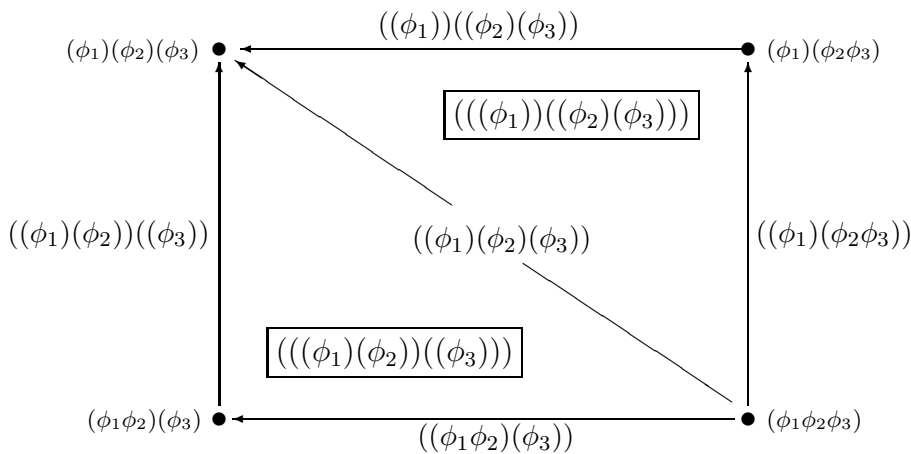
\begin{figure}[htb]
\begin{center}
%
%
%
\begin{picture}(300,170)(40,0)
%
%
\put(37,152){{\scriptsize $(\phi_{1})(\phi_{2})(\phi_{3})$}}
\put(95,155){\circle*{5}}
\put(293,155){\vector(-1,0){190}}
\put(155,161){{\small $((\phi_{1}))((\phi_{2})(\phi_{3}))$}}
\put(295,155){\circle*{5}}
\put(302,152){{\scriptsize $(\phi_{1})(\phi_{2}\phi_{3})$}}
%
%
\put(95,20){\vector(0,1){130}}
\put(15,82){{\small $((\phi_{1})(\phi_{2}))((\phi_{3}))$}}
\put(95,15){\circle*{5}}
\put(45,12){{\scriptsize $(\phi_{1}\phi_{2})(\phi_{3})$}}
%
%
\put(290,15){\vector(-1,0){190}}
\put(170,5){{\small $((\phi_{1}\phi_{2})(\phi_{3}))$}}
\put(295,15){\circle*{5}}
\put(302,12){{\scriptsize $(\phi_{1}\phi_{2}\phi_{3})$}}
%
%
\put(295,20){\vector(0,1){130}}
\put(298,82){{\small $((\phi_{1})(\phi_{2}\phi_{3}))$}}
%
%
\put(183,96){\vector(-3,2){82}}
\put(168,80){{\small $((\phi_{1})(\phi_{2})(\phi_{3}))$}}
\put(210,74){\line(3,-2){79}}
%
%
\put(180,127){\fbox{$(((\phi_{1}))((\phi_{2})(\phi_{3})))$}}
\put(113,37){\fbox{$(((\phi_{1})(\phi_{2}))((\phi_{3})))$}}
\end{picture}
\caption[fig1]{The triangulated $2$-cube \w{\FG_{3}}}
\label{fig1}
\end{center}
\end{figure}

This may be identified canonically with the standard
triangulation \w{\Delta[1]^{n}\in\Ss} of \w{I^{n}\in\C}
(see \cite[\S 3]{BBlaC}), thus indeed
identifying \w{\Fs\K} with \w[.]{T\WK}
\end{proof}

%
%
\begin{prop}\label{pwgresol}
If $\Gamma$ is a quasi-lattice, the map of \ww{\CG}-categories
\w{\gamma^{c}:\WG\to\Gamma} is a cofibrant resolution.
\end{prop}

\begin{proof}
The map of \ww{\SO}-categories \w{\gamma^{s}:\Fs\K\to\K}
is a weak equivalence, since \w{\Fs} is defined by
a comonad (see \cite[\S 1]{CPorV}). Thus \w{\Fs\K}  is indeed a free
simplicial resolution of $\K$ (see \cite[\S 2.4]{DKanS}, \cite[\S 2]{CPorV},
and \cite[\S 2.21]{BMarkH}).  Having identified
\w{\gamma^{s}:\Fs\K\to\K} with \w[,]{T\gamma^{c}:T\WK\to T\K=\K} it follows from
Proposition \ref{padjenr} that \w{\gamma^{c}} is a weak equivalence.

By construction, each composition map
\w{\WK(a,b)\otimes\WK(b,c)\to\WK(a,c)} of \w{\WK} is an inclusion of
a sub-cubical complex, since on every ``product'' cube
\w{\II{n}{\fd}\otimes\II{k}{\gd}\cong\II{n+k}{\fd\otimes\gd}\subseteq
\II{n+k+1}{\fd\circ\gd}} it is the inclusion of a $1$-face.
Thus the map \w{\varphi\uv:\colim D^{\ast,\B}\uv\to\B^{c}(u,v)} of Definition
\ref{dcdiagram} is just the inclusion of the sub-cubical complex
consisting of all the $1$-faces, which is a cofibration (in fact, an
anodyne map). This shows that \w{\WK} is cofibrant.
\end{proof}

\begin{mysubsect}{Rectifying homotopy commutative diagrams}\label{srhcd}

We can use the cofibrant resolution \w{\WK\to\K} to study the
rectification of a homotopy-commutative diagram
\w{\tX:\K\to\ho\M} in some model category $\M$ (such as $\TT$ or \w[).]{\Ta}

Since the $0$-skeleton of \w{\WK} is isomorphic to \w[,]{F\K}
choosing an arbitrary representative \w{X_0(f)} for each homotopy class
\w{\tX(f)} for each morphism $f$ of $\K$, yields
a lifting of $\tX$ to \w[.]{X_{0}:\skc{0}\WK\to\M}

Note that a choice of a $0$-realization \w{X_{0}:F\K\to\M}
is equivalent to choosing basepoints in each relevant component
of each \w[,]{\M^{c}(u,v)} although
of course this cannot be done coherently unless $\tX$ is rectifiable.
\end{mysubsect}

\begin{remark}\label{rskel}
Our goal is to extend \w{X_{0}} over the skeleta of \w[.]{\WK}
However, the ``naive'' cubical skeleton functor \w{\skc{k}:\C\to\C}
(\S \ref{dcube}) is not monoidal with respect to $\otimes$ (unlike the
simplicial analogue), so it does not commute with composition maps.
Nevertheless, one can define a $k$-skeleton functor for
\ww{\CO}-categories in general; when $\Gamma$ is a quasi-lattice (\S
\ref{dpi}) and $\A$ is a cofibrant \ww{\CG}-category (such as \w[),]{\WG}
\w{\skc{k}\A} can be defined by simply including all $\otimes$-product
cubes of $i$-cubes in $\A$ with \w[.]{i\leq k} Of course, if $\A$
is $n$-dimensional (that is, has no non-degenerate $i$-cubes for
\w[),]{i>n} then \w{\skc{n}\A=\A} agrees with the naive $n$-skeleton.
\end{remark}

If $\M$ is cubically enriched (\S \ref{dcmap}),
extending \w{X_{0}} to a cubical functor \w{X_{1}:\skc{1}\WK\to\M} is
equivalent to choosing homotopies between each $\tX(f_1 \circ f_2)$ and
$\tX(f_1) \circ \tX(f_2)$, since the $1$-cubes of \w{\WK} correspond to
all possible (two term) factorizations of maps in $\K$. Extending \w{X_{1}}
further to \w{X_{2}:\skc{2}\WK\to\M} means choosing homotopies between the
homotopies for three-fold compositions, and so on.

This is the idea underlying a fundamental result of Boardman and Vogt:

\begin{thm}[\protect{\cite[Cor.~4.21 \& Thm.~4.49]{BVogHI}}]\label{tbv}
A diagram \w{\tX:\K\to\ho\TT} lifts to $\TT$ if and only if it extends
to a simplicial functor \w[.]{X_{\infty}:\WK\to\TT}
\end{thm}

\begin{remark}\label{rphhod}
In fact, for our purposes we do not have to assume that the
category $\M$ is cubically enriched, or even has a model category
structure: all we need is for $\M$ to have a suitable class of weak
equivalences $\bW$, from which we can construct an \ww{\SO}-category
\w{L(\M,\bW)} as in \cite[\S 4]{DKanF}, and then the corresponding
\ww{\CO}-category \w{\Sb L(\M,\bW)} by Proposition \ref{padjenr}.
Note that when $\M$ and $\bW$ are pointed, the construction of Dwyer
and Kan is naturally pointed, too. However, to avoid excessive verbiage
we shall assume for simplicity that $\M$ is a cubically enriched
model category.
\end{remark}

We do not actually need the full (usually large) category $\M$
(or \w[),]{\Sb L(\M,\bW)} since we can make use of the following:

\begin{defn}\label{dcx}
Given a diagram \w{\tX:\K\to\ho\M} for a model
category \w[,]{\M\in\CC} let \w{\CX} be the smallest
\ww{\CoK}-category inside $\M$ through which any lift of $\tX$ to
\w{X:\K\to\M} factors. This means that \w{\CX} is the
\ww{\CoK}-category having cubical mapping spaces
$$
\CX(X u,X v)~:=~\begin{cases}
\w{\M^{c}(X u,X v)} & \text{~if~~} u\prec v \text{~~in~}\OO:=\Obj\K\\
\emptyset & \text{~otherwise.}\end{cases}
$$
This is a sub-cubical category of $\M$.

For simplicity, we further reduce the mapping spaces of \w{\CX} so that they
consist only of those components of \w{\M^{c}(X u,X v)} which are
actually hit by $\tX$, so that \w[.]{\pi_{0}\CX=\K}
In particular, if $\K$ is the partially ordered set
\w[,]{\lra{\OO,\prec}} we may assume the mapping spaces of \w{\CX} are
connected (when they are not empty).
\end{defn}

\begin{mysubsection}{Pointed diagrams}\label{spointedd}
We want to understand the relationship between two possible ways to
describe the (final) obstruction to the existence of
an extension \w[:]{X_{\infty}} topologically and cohomologically.
Unfortunately, even though these obstructions can be defined for quite
general $\K$, they do not always coincide; this can be seen by
comparing the sets in which they take value.

However, we are in fact only interested in the cases where the
obstruction can naturally be thought of as the \emph{higher homotopy
operation} associated to the data \w[.]{\tX:\K\to\ho\TT} The usual
mantra says that such an operation is defined when ``a lower order
operation vanishes for two (or more) reasons''. Indeed, the example
of the usual Toda bracket shows that the problem cannot be stated
simply in terms of rectifying a homotopy-commutative diagram, since
any diagram indexed by a linear indexing category \w{\Gamma_{n}} as
above can always be rectified: what we want is to realize certain
null-homotopic maps by \emph{zero} maps (see \cite[\S 3.12]{BMarkH}).

This suggests that we restrict attention to \emph{pointed} diagrams,
and to the following special type of indexing category:
\end{mysubsection}

\begin{defn}\label{dlatt}
A \emph{lattice} is a finite quasi-lattice $\Gamma$ (\S \ref{dquasilat})
equipped with a \emph{(weakly) initial} object \w{\vi} and a
\emph{(weakly) final} object \w[,]{\vf} satisfying:
\begin{enumerate}
\renewcommand{\labelenumi}{(\alph{enumi})~}
\item There is a unique \ww{\phi_{\max}:\vi\to\vf}.
\item For each \w[,]{v\in\Obj\Gamma} there is at least one map
\w{\vi\to v} and at least one map \ww{v\to\vf}.
\end{enumerate}

A composable sequence of $n$ arrows in $\Gamma$ will be called an
$n$-\emph{chain}. The maximal occuring $n$ (necessarily for a chain
from \w{\vi} to \w[,]{\vf} factorizing \w[)]{\phi_{\max}} is called
the \emph{length} of $\Gamma$.
\end{defn}

\begin{remark}\label{rdim}
Note that if the length of $\Gamma$ is \w[,]{n+1} then \w{\WG} is
$n$-dimensional, in the sense that the cubical function complex
\w{\WG(\vi,\vf)} has dimension $n$, and \w{\dim(\WG(u,v)) < n}
for any other pair \w{u,v} in $\Gamma$.
\end{remark}

\begin{defn}\label{dnullcube}
We shall mainly be interested in the case when $\Gamma$ is pointed
(in which case necessarily \w[).]{\phi_{\max}=0} A
\emph{null sequence} in $\Gamma$ is then a composable sequence
$$
\fd~:=~\left(a_{n+1}\xra{f_{n+1}}a_{n}\xra{f_{n}}a_{n-1}\dotsc
a_{1} \xra{f_{1}} a_{0}\right)
$$
with \w[,]{\comp{\fd}=0} but no constituent \w{f_{i}} is zero.
It is called \emph{reduced} if all adjacent compositions
\w{f_{i+1}\cdot f_{i}} ($i=1,\dotsc,n$) are zero.
An $n$-cube \w{\II{n}{\fd}} in \w{\WG} indexed by a (reduced) null
sequence is called a (reduced) \emph{null cube}.
\end{defn}

As noted above, we want to concentrate on the problem of
replacing null-homotopic maps with zero maps, given a pointed diagram
\w{\tX:\Gamma\to\ho\M} which commutes up to pointed homotopy.
We shall therefore assume from now on that all other (non-zero) triangles
in the diagram commute strictly. However, since the non-zero maps in
$\Gamma$ do not form a sub-category, we shall need the following:

\begin{defn}\label{dpoint}
The \emph{unpointed version} \w{\UP{\K}} of a pointed category $\K$
is defined as follows: if \w{\K\cong F(\K)/I} for some set of
relations $I$ in the free category \w[,]{F(\K)} then the objects of
\w{\UP{\K}} are those of $\K$, except for the zero objects,
and \w[,]{\UP{\K}:=\F(K')/(I\cap\F(K'))} where \w{K'} is obtained from the
underlying graph $K$ of $\K$ by omitting all zero objects and maps. The
inclusion \w{K'\hra K} induces a functor \w[.]{\iota:\UP{\K}\to\K}

Essentially, \w{\UP{\K}} is the full subcategory of $\K$ omitting $0$
and all maps into or out of the zero object $0$. However, if the composite
\w{f\cdot g:a\to b}  is zero in $\K$ with \w[,]{f\neq 0\neq g} then we
add a new (non-zero) map  \w{\varphi:a\to b} in \w{\UP{\K}} (with
\w[),]{\iota(\varphi)=0} to serve as the composite in \w{\UP{\K}} of
$f$ and $g$.
\end{defn}

\begin{mysubsect}{Defining higher operations}\label{sdpho}

{F}rom now on we assume given a pointed lattice $\Gamma$ and a diagram
up-to-homotopy \w{\tX:\Gamma\to\ho\M} into a pointed cubically enriched
model category $\M$. Setting \w[,]{\Gp:=\UP{\Gamma}} we also assume that
the composite \w{\tX\circ\iota} lifts to a strict diagram \w[.]{X':\Gp\to\M}
For simplicity we also denote the factorization of \w{X'} though \w{\CX}
(\S \ref{dcx}) by \w[.]{X':\Gp\to\CX}

Our goal is to extend \w{X'} to a pointed diagram \w[.]{X:\Gamma\to\CX}
(Note that \w{X'} itself cannot be pointed in our sense, but it
still takes values in the pointed category $\M$).
Obviously, if \w{X'} does extend to such an $X$, every map
\w{\varphi\in\Gp} which factors through $0$ in $\Gamma$ must be
(weakly) null-homotopic in $\M$.  Thus, we additionally include this
restriction on the original data as part of our assumptions.

Our approach is to extend \w{X'} by induction over the skeleta of \w[,]{\WG}
where we actually need:
\end{mysubsect}

\begin{defn}\label{drelskel}
Given $\Gamma$ and \w{X'} as above, for each \w{k\geq 0} the
\emph{relative $k$-skeleton} for \w[,]{(\Gamma,\Gp)} denoted by
\w[,]{\skG{k}} is the pushout:
$$
\xymatrix@R=25pt{
\skc{k}\WG' \ar[r]^{\skc{k}\iota} \ar[d]^{\skc{k}\gamma^{c}} & \skc{k}\WG \ar[d] \\
\Gp \ar[r] & \skG{k}
}
$$
\noindent in \w{\CGC} (cf.\ Lemma \ref{llimit}), where
\w{\gamma^{c}:\WG\to\Gamma} is the augmentation of Proposition \ref{pwgresol}.

Note that the natural inclusions \w{\skc{k-1}\hra\skc{k}} induce maps
\w[.]{\skG{k-1}\to\skG{k}} A map of \ww{\CG}-categories \w{X'_{k}:\skG{k}\to\CX}
extending \w{X':\Gp\to\CX} is called $k$-\emph{allowable}.

In particular, if $\Gamma$ is a lattice of length \w[,]{n+1}
by Remark \ref{rdim} \w{\WGGp:=\skG{n}} is the pushout
$$
\xymatrix@R=25pt{
\WG' \ar[r]^{\iota} \ar[d]^{\gamma^{c}} & \WG \ar[d] \\
\Gp \ar[r] & \WGGp~.
}
$$
\end{defn}

\begin{remark}\label{rrelskel}
\w{X'} extends canonically to a pointed map \w[,]{X_{0}:\skc{0}\WG\to\CX}
because \w{\skc{0}\WG} is a free category, and the only new object is $0$.
Together with \w{\gamma^{c}} this determines a canonical
$0$-allowable extension \w[.]{X'_{0}:\skG{0}\to\CX}

If $\Gamma$ is a lattice of length \w[,]{n+1}
in order to rectify \w{X'} we want to extend \w{X'_{0}}
inductively over the relative skeleta \w{\skG{k}} to an
$n$-allowable map \w{X'_{\infty}:\WGGp\to\CX} \wh equivalently, a
map \w{X_{\infty}:\WG\to\CX} which agrees with the initial
\w[.]{X':\Gp\to\CX} Recall that because \w[,]{\dim\WG=\dim\WGGp=n}
\w{X'_{n}} is actually \w{X'_{\infty}} in the sense of Theorem
\ref{tbv}, so this yields a rectification of \w{X'} for suitable
$\M$ (such as \w[).]{\Ta}

We assumed in \S \ref{sdpho} that \w{X':\Gp\to\CX} takes every map
\w{\varphi\in\Gp} which factors through $0$ in $\Gamma$ to one which is
null-homotopic in $\M$. Therefore, by choosing null-homotopies for all
such maps we see that \w{X'_{0}} always extends \emph{non-canonically} to a
$1$-allowable \w[.]{X'_{1}:\skG{1}\to\CX}

However, in general there are obstructions to obtaining $k$-allowable
extensions for \w[.]{k\geq 2} These are complicated to define
``topologically'' (see \cite{BMarkH} and \cite{BChachP}). Fortunately,
in order to define the higher homotopy operation associated to
\w[,]{X'} we only need to consider the \emph{last} obstruction.

That is, we assume we have already produced an \ww{(n-1)}-allowable extension
\w[,]{X'_{n-1}:\skG{n-1}\to\CX} and want to extend it to \w[.]{X'_{n}}
It may be possible to do so in different ways. In order to define the
set \w{\llrr{X'}} of ``last obstructions'',  we need the following:
\end{remark}

\begin{lemma}\label{llinex}
Assume that \w{\Gamma=\Gamma_{n+1}} is a composable \ww{(n+1)}-chain \w{\fd}
(\S \ref{eglatt}) and that the $i$-th adjacent composition
\w{f_{i}\cdot f_{i+1}\neq 0} in $\Gamma$, and let
$$
\fd':=(f_{1},\dotsc,f_{i-1},f_{i}\cdot f_{i+1},f_{i+2},\dotsc f_{n})~.
$$
Let \w{\iota:\II{n}{\fd'}\hra\II{n+1}{\fd}} be the inclusion of the
$i$-th zero face. Let $\tG$ be the linear lattice corresponding to
\w[.]{\fd'} Then for any \w[,]{X:\Gp\to\CX} the inclusion
\w{\iota:\tG\hra\Gamma} induces a one-to-one correspondence between
the set of extensions of \w{X'_{0}:\skG{0}\to\CX} to \w{\WG} and the
extensions of \w{\tX'_{0}:\sktG{0}\to\CX} to \w[.]{W\tG}
\end{lemma}

\begin{proof}
The $i$-th dimension of \w{\II{n}{\fd}} corresponds to the $i$-th
adjacent composition \w{f_{i}\cdot f_{i+1}} in the \ $(n+1)$-chain
\w[,]{\fd} and if this composite is not zero, then \w[,]{X'_{n}}
being allowable, is constant along this dimension. Thus the projection
\w{\rho:\II{n}{\fd}\to \II{n-1}{\fd'}} induces the inverse to
\w[.]{\iota^{\ast}}
\end{proof}

\begin{prop}\label{pphho}
Let $\Gamma$ be a lattice of length \w{n+1} and \w{X':\Gp\to\CX} a diagram.
Let \w{\JG} be the set of length \w{n+1} reduced null sequences of $\Gamma$
(Definition \ref{dnullcube}). There is a natural correspondence between
\ww{(n-1)}-allowable extensions \w{X'_{n-1}:\skG{n-1}\to\CX} of \w{X'}
and maps \w[,]{F_{X'_{n-1}}:\bigvee_{\fd\in\JG}~\Sigma^{n-1}X'(\vi)\to X'(\vf)}
such that \w{F_{X'_{n-1}}} is null-homotopic if and only if \w{X'_{n-1}}
extends to \w[.]{\skG{n}}
\end{prop}

\begin{proof}
In order to extend \w{X'_{n-1}:\skG{n-1}\to\CX} to \w[,]{\skG{n}} we must
choose extensions to the $n$-cubes of \w[.]{\WG} These occur only in the full
mapping complex \w[,]{\WG(\vi,\vf)} and are in one-to-one correspondence
with those decompositions
$$
\fd~=~\left(\vi=a_{n+1}\xra{f_{n+1}}a_{n}\xra{f_{n}}a_{n-1}\dotsc
a_{1}\xra{f_{1}}a_{0}=\vf\right)
$$
of \w{\phi_{\max}:\vi\to\vf} which are of maximal length \w[.]{n+1}
Note that the minimal vertex of \w{\II{n}{\fd}} is indexed by
\w[;]{\phi_{\max}=0} the maximal vertex is \w{\II{0}{\otimes \fd}}
(Definition \ref{dvertices}).

By Lemma \ref{llinex} we need only consider those maximal
decompositions \w{\fd} for which every adjacent composition
\w[.]{f_{i}\cdot f_{i+1}=0} In this case, we may assume that any facet
\w{\II{n-1}{\fd'}} of \w{\II{n}{\fd}} which touches the vertex labeled
by \w{\phi_{\max}=0} has at least one factor of \w{\fd'} equal to $0$
(in $\Gamma$), so  \w[.]{X'_{n-1}\rest{\II{n-1}{\fd'}}=0} Thus
\w{X'_{n-1}\rest{\II{n}{\fd}}} is given by a map in $\M$
\w{F'_{(X'_{n-1},\II{n}{\fd})}:X'(\vi)\otimes \partial I^{n}\to X'(\vf)}
which sends \w{X'(\vi)\otimes \II{0}{\phi_{\max}}} and
\w{\ast_{X(\vi)}\otimes I^{n}} to \w[,]{\ast_{X(\vf)}} so it induces
\w[.]{\tilde{F}_{(X'_{n-1},\II{n}{\fd})}:X'(\vi)\wedge S^{n-1}\to X'(\vf)}

Note further that any two such $n$-cubes \w{\II{n}{\fd}} and
\w{\II{n}{\gd}} have distinct maximal vertices \w{\II{0}{\fd}} and
\w[,]{\II{0}{\gd}} so they can only meet in facets adjacent to the
minimal vertex, where \w{\tilde{H}} vanishes. Thus altogether
\w{X'_{n-1}} is described by a map
\begin{myeq}\label{eqrednullcube}
F_{X'_{n-1}}:\bigvee_{\fd\in\JG}~\Sigma^{n-1}X'(\vi)~\to~X'(\vf)~,
\end{myeq}
\noindent where \w{\JG} is the set of length \w{n+1} reduced null
sequences of $\Gamma$. Clearly, \w{F_{X'_{n-1}}} is null-homotopic if and
only if \w{X'_{n-1}} extends to all of \w[,]{\WG}
since \w{\WG(\vi,\vf)} is
\w[,]{\map(C(\bigvee_{\fd\in\JG}~\Sigma^{n-1}X'(\vi)),\,X'(\vf))} up to
homotopy, where \w{CK} is the cone on $K$.
\end{proof}

\begin{defn}\label{dphho}
The $n$-th order \emph{pointed higher homotopy operation} \w{\llrr{X'}}
associated to \w{X':\Gp\to\CX} as above is defined to be the subset:
\begin{myeq}\label{eqphho}
\w{\llrr{X'}}~\subseteq~
\left[\bigvee_{\fd\in\JG}~\Sigma^{n-1}X'(\vi),~X'(\vf)\right]_{\ho\M}
\end{myeq}
\noindent consisting of all maps \w{F_{X'_{n-1}}} as above, for all possible
choices of \ww{(n-1)}-allowable extensions \w[,]{X'_{n-1}} of \w[.]{X'}
We say the operation \emph{vanishes} if this set contains the zero class.
\end{defn}

%
%
\sect{Cohomology and rectification}
\label{ccoh}

The approach of Dwyer, Kan, and Smith to realizing a homotopy-commutative
diagram \w{\tX:\Gamma\to\ho\M} is also based on Theorem \ref{tbv},
which says that $\tX$ can be rectified if and only if it extends to
\w[.]{\WG} We do not actually need the full force of their theory, which is
why we can work in an arbitrary pointed model category $\M$, rather than just
\w{\Ta} (see also Remark \ref{rphhod}).

Essentially, they define the (possibly empty) moduli space \w{\hc{}\tX}
to be the nerve of the category of all possible rectifications of $\tX$
(cf.\ \cite[\S 2.2]{DKSmH}), and \w{\hc{\infty}\tX} is the space of all
$\infty$-homotopy commutative lifts of $\tX$ in (the simplicial version of)
\w{\map_{\CC}(\WG,\M)=\map_{\CGC}(\WG,\CX)} (\S \ref{dcx}). They then show that
\w{\hc{}\tX} is (weakly) homotopy equivalent to \w{\hc{\infty}\tX}
(see \cite[Theorem 2.4]{DKSmH}).
Thus the realization problem is equivalent to finding suitable elements in
\w[.]{\map_{\CGC}(\WG,\CX)} Dwyer, Kan, and Smith also consider a relative
version, where $\tX$ has already been rectified to \w{Y:\Theta\to\M} for some
sub-category \w{\Theta\subseteq\Gamma} (see \cite[\S 4]{DKSmH}).
We shall in fact need only the case \w{\Theta=\Gp} and \w[,]{Y=X'} so we want
an element in \w{\map_{\CGC}(\WGGp,\CX)} (see \S \ref{drelskel}).

\begin{mysubsection}{The tower}\label{sdkstower}
If $\Gamma$ is a quasi-lattice, \w{\CGC} has a semi-spherical model
category structure (see \S \ref{smccgc} and \S \ref{ssscgc}). Therefore, the
Postnikov tower \w{\{P^{m}\CX\}_{m=0}^{\infty}} of the \ww{\CG}-category
\w{\CX} allows us to define \w{\hc{m}\tX:=\map_{\CGC}(\WGGp,P^{m-1}\CX)} for
\w[.]{m\geq 1} Note that \w{P^{0}\CX} is homotopically trivial \wh that is, each
component of each mapping space \w{(P^{0}\CX)(u,v)} is contractible \wh so
\w{\hc{1}\tX} is, too. Moreover, \w{\tX:\Gamma\to\ho\M}
(or \w[)]{X':\Gp\to\CX} determines a canonical ``tautological'' component
of \w{\hc{1}\tX} \wh namely, the component of the map
\w[,]{\widetilde{X_{1}}:\WGGp\to P^{0}\CX} corresponding to the canonical
$0$-allowable extension \w{X'_{0}:\skG{0}\to\CX} of \S \ref{rrelskel}.
\end{mysubsection}

Because \w{\CX} is weakly equivalent to the limit of its Postnikov tower
(\S \ref{sfstr}(b)), the space \w{\hc{\infty}\tX} is the homotopy limit
of the tower:
%
\begin{myeq}\label{eqthirteen}
\hc{\infty}\tX \to \dotsc \to \hc{n}\tX \to \hc{n-1}\tX \dotsc \to \hc{1}\tX~.
\end{myeq}

In general, there are \w{\lim^{1}} problems in determining the components of
\w{\hc{\infty}\tX} (see \cite[\S 4.8]{DKSmO}), but these will not be relevant
to us here, because of the following:

%
%
\begin{lemma}\label{cfinite}
If $\Gamma$ has length \w[,]{n+1} the tower \wref{eqthirteen} is constant
from \w{\hc{n-1}\tX} up.
\end{lemma}

\begin{proof}
We may assume that \w{\CX} is fibrant (e.g., if \w{\tX v} is a cubical Kan
complex for each \w[).]{v\in\OO} Then \w{\skG{n}=\WG} by Remark \ref{rdim},
where in this case we are using the naive $n$-skeleton (see Remark \ref{rskel})
which is left adjoint to the $n$-coskeleton functor. By Remark \ref{rcsk},
we may use the latter for \w[.]{P^{n-1}\CX} Thus the choices of $n$-allowable
extensions \w{X'_{n}:\skG{n}=\WG\to\CX} of $\tX$ are in natural one-to-one
correspondence with lifts \w{\widetilde{X_{n}}:\WGGp\to P^{n-1}\CX}
of \w[.]{\widetilde{X_{1}}}
\end{proof}

\begin{mysubsect}{The obstruction theory}\label{sobstthy}

In view of the above discussion, the realization problem for
\w{\tilde{X}:\Gamma\to\ho\M} \wh and in particular, the pointed version for
\w{X':\Gp\to\M} (see \S \ref{sdpho}) \wh can be solved if one can successively
lift the element \w{\widetilde{X_{1}}\in\hc{1}\tX} through the tower
\wref[.]{eqthirteen} In fact, we do not really need the (simplicial or cubical)
mapping spaces \w{\hc{m}\tX:=\map_{\CGC}(\WGGp,P^{m-1}\CX)} at all \wh
we simply need to lift the maps \w{\widetilde{X_{m}}:\WGGp\to P^{m-1}\CX}
in the Postnikov tower for \w[.]{\CX}

Let \w{k_{m-1}:\CX\to E^{G}(\pi_{m}\CX,m+1)} be the \ww{(m-1)}-st $k$-invariant
for \w[,]{\CX} where \w{G:=\hpi{\CX}} (see \S \ref{ssscgc} \textit{ff.}).
Given a lifting \w[,]{\widetilde{X_{m}}} composing it with \w{k_{m-1}}
yields a map \w[:]{h(X_{m}):\WGGp\to E^{G}(\pi_{m}\CX,m+1)}
$$
\xymatrix@R=25pt{
\WGGp \ar@/^2pc/[drr]^{\widetilde{X_{m}}}
\ar@{.>}[dr]^{\widetilde{X_{m+1}}} \ar@/_2pc/[ddr]_{p} \\
& {P^{m}\CX} \ar[r] \ar[d] & {P^{m-1}\CX} \ar[d]^{k_{m-1}} \\
& BG \ar[r]^-{s} & {E^{G}(\pi_{m}\CX,m+1)} \ar@/^1pc/[l]^{\text{proj}}
}
$$

To identify \w{h(X_{m})} as an element in the appropriate cohomology group
(Definition \ref{dcoh}), note that in this case the twisting map
\w{p:\WGGp\to BG} factors through
\w[,]{\hpi{X_{n}}:\hpi{\WGGp}\to\hpi{P_{m-1}\CX}=\hpi{\CX}=G}
and by Proposition \ref{pwgresol}, the fundamental groupoid
\w{\hpi{\WGGp}=\Gamma} is discrete. Thus \w{[h(X_{m})]} takes value
in \w[,]{H^{m+1}_{\Gamma}(\WGGp;\pi_{m}\CX)} which we abbreviate
to \w[.]{H^{m+1}(\Gamma;\pi_{m}\CX)}

The lifting property for a fibration sequence (over \w[)]{BG} then yields:
\end{mysubsect}

%
%
\begin{prop}[\protect{\cite[Prop.\ 3.6]{DKSmH}}]\label{pdksm}
The map \w{\widetilde{X_{m}}} lifts to \w{\widetilde{X_{m+1}}} in
\w{\hc{m+1}\tX} if and only if \w{[h(X_{m})]} vanishes in
\w[.]{H^{m+1}(\Gamma;\pi_{m}\CX)}
\end{prop}

\begin{mysubsect}{Relating the two obstructions}\label{srto}

In order to see how the two obstructions we have described are related,
we need some more notation:

For a pointed lattice $\Gamma$ of length \w[,]{n+1} let \w{\hWG}
denote the sub-$\CG$-category of \w{\WGGp} obtained from
\w{\skG{n-1}} by adding all unreduced null \ww{n}-cubes
(Definition \ref{dnullcube}). By Lemma \ref{llinex}, any  \ww{(n-1)}-allowable
extension \w{X'_{n-1}:\skG{n-1}\to\CX}  extends canonically
to \w[.]{\hX:\hWG\to\CX}
If \w{i_{n}:\sk{n}\hWG\to\hWG} and \w{i:\hWG\to\WGGp} are the inclusions,
we thus have a commutative diagram in \w[:]{\CGC}
$$
\xymatrix@R=25pt{
\skc{n-1}\hWG \ar[d]_{\skc{n-1}{i}=\Id} \ar[rr]^{i_{n-1}} \ar[rrd]^{X'_{n-1}} & &
\hWG \ar[d]^{\hX} \\  \skG{n-1}\ar[rr]_{X'_{n-1}} & & \CX
}
$$
\noindent Because $\Gamma$ is a lattice of length
\w[,]{n+1} \w{\WG} is $n$-dimensional. Furthermore, if we break up any chain in $\Gamma$ into
disjoint sub-chains of length $k$ and $\ell$ \wb[,]{k+\ell=n+1} the resulting composite cube has
dimension \w[.]{(k-1)+(\ell-1)=n-1} Thus the only non-degenerate $n$-cubes in
\w{\WG} are indecomposable in \w[,]{\WG(\vi,\vf)} which implies
that \w{\skG{n-1}} is in fact defined using the naive
\ww{(n-1)}-skeleton (see Remark \ref{rskel}).

Thus by adjointness (using Remark \ref{rcsk}) we have:
%
\mydiagram[\label{eqfourteen}]{
\hWG \ar[d]_{i} \ar[rr]^{\hX} & & \CX \ar[d]^{r} \\
\WGGp \ar[rr]_<<<<<<<<<<{\widetilde{X_{n}}} & & \cskc{n-1}\CX=P_{n-2}\CX
}
\noindent in which $r$ is the fibration \w{r\q{n-1}=p\q{n-1}} of
\S \ref{sfstr}(b).

Now let \w{\RG} be the \ww{\CG}-category of all \emph{reduced} null
\ww{(n-1)}-spheres (that is, boundaries of the reduced null $n$-cubes)
in \w[.]{\WG} Thus:
%
\begin{myeq}\label{eqfifteen}
\RG^{c}(u,v)=\begin{cases}
\bigcup_{\fd\in\JG}~\partial \II{n}{\fd} &~~\text{if}~~(u,v)=(\vi,\vf)\\
\emptyset &~~\text{otherwise}\end{cases}
\end{myeq}
\noindent (in the notation of \wref[)]{eqrednullcube}.
\end{mysubsect}

\begin{fact}\label{fcofib}
There is a homotopy cofibration sequence of \ww{\CG}-categories
\end{fact}
%
\begin{myeq}\label{eqsixteen}
\RG~\xra{j}~\hWG ~\xra{i}~\WGGp~.
\end{myeq}

\begin{proof}
By definition of a pointed lattice, all the $n$-cubes of \w{\WG} (and thus of
\w[)]{\WGGp} are null cubes. Thus the map \w{i:\hWG\to\WGGp} is actually
an isomorphism in all mapping slots except \w[,]{(u,v)=(\vi,\vf)} where
the $n$-cells attached via $j$ provide the missing (necessarily reduced)
null $n$-cubes.
\end{proof}

\begin{defn}\label{dcorresp}
Let $\Gamma$ be a pointed lattice  of length \w[,]{n+1} \w{\CX} a
\ww{\CG}-category, and define \w{\JG} as in Proposition \ref{pphho}.
To each commuting square:
\mydiagram[\label{cohsquare}]{
\hWG \ar[d]_{i} \ar[r]^{\hat{h}} &  \CX \ar[d]^{r} \\
\WGGp \ar[r]_{h} & P_{n-2}\CX
}
\noindent in \w[,]{\CGC} we assign the composite \w{k_{n-2}\cdot h} in
\w[.]{H^{n}(\Gamma,\pi_{n-1}\CX)} Denote by \w{K_{n}(\CX)} the subset of
\w{H^{n}(\Gamma,\pi_{n-1}\CX)} consisting of all such elements
\w[.]{k_{n-2}\cdot h} Finally, define
\w{\Phi_{n}:K_{n}(\CX)\to\prod_{\fd\in\JG}~\pi_{n-1}\CX(\vi,\vf)} by
assigning to \wref{cohsquare} the homotopy class of the composite
\w[.]{\sigma:=(\hat{h}\cdot j)(\vi,\vf):\RG(\vi,\vf)\to\CX(\vi,\vf)}
\end{defn}

\begin{lemma}\label{lcorresp}
The map \w{\Phi_{n}} is well-defined.
\end{lemma}

\begin{proof}
Freudenthal suspension gives an isomorphism
$$
[\RG(\vi,\vf),\CX(\vi,\vf)]_{\ho\C}~\xra{\cong}~
[\Sigma\RG,E^{\hpi{\WG}}(\pi_{n-1}\CX,n)]_{\ho\CGC}~,
$$
so \w{\Phi_{n}} may be equivalently defined by assigning to the composite
\w{k_{n-2}\cdot h} the extension \w{e=\Sigma\sigma} in the following diagram:
$$
\xymatrix@R=25pt{
\hWG \ar[d]_{i} \ar[rr]^{\hat{h}} & & \CX \ar[d]^{p} \\
\WG \ar[d]_{\partial} \ar[rr]^{h}\ar[rrd]^<<<<<<<<<<<<<<<<<<<<{k_{n-2}\cdot h} & &
P_{n-2}\CX \ar[d]^{k_{n-2}} \\
\Sigma\RG \ar@{.>}[rr]_<<<<<<<<<{e} & & E^{G}(\pi_{n-1}\CX,n)
}
$$
\noindent where \w{\WG\xra{\partial}\Sigma\RG\xra{\Sigma{j}}\Sigma\hWG}
is the continuation of the cofibration sequence of \wref[.]{eqsixteen}
Here we used the fact that \w{\RG} is concentrated in the \w{(\vi,\vf)}
slot, by \wref[.]{eqfifteen}

Note that the extension $e$ (and thus \w[,]{\sigma=\Phi_{n}(k_{n}\cdot h)}
the adjoint of $e$ with respect to the \w{(\Sigma,\Omega)}
adjunction) is uniquely determined up to homotopy, since
\w{[\Sigma\hWG,E^{\hpi{\WG}}(\pi_{n}\CX,n+1)]=0} for dimension reasons.
\end{proof}

Our main result, Theorem A of the Introduction, is now a consequence
of the following Theorem and Corollary:

%
%
\begin{thm}\label{tone}
Given \w{X':\Gp\to\M} as in \S \ref{sdpho}, the map \w{\Phi_{n}} is a
\emph{pointed} correspondence between the set of elements of
\w{K_{n}(\CX)} obtained from commuting squares of the form
\wref{eqfourteen} and \w{\llrr{X'}} of \wref{eqphho} \wh that is,
\w{\Phi_{n}(\alpha)=0} if and only if \w[.]{\alpha=0}
\end{thm}

\begin{proof}
By Proposition \ref{pdksm}, the composite
\w{h(X_{n-1}):=k_{n-2}\cdot\widehat{X_{n-1}}} is the obstruction to
extending \w{\hX} to \w[,]{X_{n}:\WGGp\to\CX} and since
$$
\xymatrix@R=25pt{
\RG \ar[d]_{j} \ar[rr]^{\sigma} & & E^{G}(\pi_{n-1}\CX,n-1)\ar[d]^{\ell}\\
\hWG \ar[d]_{i} \ar[rr]^{\hX} & & \CX \ar[d]^{p} \\
\WG \ar[d]_{\partial}
\ar[rr]_{\widehat{X_{n-1}}}\ar@{.>}[rru]_<<<<<<<<<<<<<<<<<<<<{X_{n}} & &
P_{n-2}\CX \ar[d]^{k_{n-2}} \\
\Sigma\RG \ar[rr]_<<<<<<<<<<{e} & & E^{G}(\pi_{n-1}\CX,n)
}
$$
\noindent commutes, with the left vertical column a cofibration and
the right vertical column a fibration sequence, the fact that
\w{e=0~\Leftrightarrow \sigma=0} implies that the composite
\w[.]{0=e\cdot\partial=k_{n-2}\cdot\widehat{X_{n-1}}=h(X_{n-1})}
Conversely, if \w{X_{n}} exists, then \w[,]{\hX\cdot j=X_{n}\cdot i\cdot j=0}
so \w[,]{\ell\cdot\sigma=0} and since \w{\pi_{n-1}\ell} is an
isomorphism, \w[.]{\sigma=0}
\end{proof}

%
%
\begin{cor}\label{ccorresp}
The Dwyer-Kan-Smith obstruction class \w{[h(X_{n-1})]} of  Proposition
\ref{pdksm} is zero in \w{H^{n}_{\Gamma}(\WG;\pi_{n-1}\CX)} if and
only if the corresponding homotopy class \w{F_{X'_{n-1}}} is
null. Therefore, \w{\llrr{X'}} vanishes if and only if \w{K_{n}(\CX)}
contains $0$.
\end{cor}

\begin{remark}\label{rcorresp}
Evidently, both the classes \w{[h(X_{n-1})]} and the set \w{\llrr{X'}} serve
as obstructions to rectifying \w[,]{\tX:\Gamma\to\ho\M} given the unpointed
rectification \w[.]{X':\Gp\to\M} It is therefore clear that they must
``vanish'' simultaneously. The point of our analysis is to give an
explicit correspondence between the individual elements of
\w{\llrr{X'}} and cohomology classes in
\w[.]{H^{n}_{\Gamma}(\WG;\pi_{n-1}\CX)} By thus describing higher
homotopy operations in cohomological terms, we may hope to use
algebraic methods to study and calculate them.
\end{remark}

\end{document}